\documentclass[12pt]{amsart}
\usepackage{geometry}
\usepackage{microtype}
\usepackage{latexsym,amsthm,amssymb,amsmath,mathrsfs, dsfont}
\usepackage{enumerate}
\usepackage{xcolor}
\usepackage{bm}
\usepackage[T1]{fontenc}
\usepackage[anchorcolor=blue,%
 bookmarks=true,%
 bookmarksnumbered=true,%
]{hyperref}
\usepackage{graphicx}

\newtheorem{thm}{Theorem}[section]
\newtheorem{prop}[thm]{Proposition}
\newtheorem{cor}[thm]{Corollary}
\newtheorem{lem}[thm]{Lemma}

 \newtheorem*{ack}{Acknowledgment}
 
\newcommand{\thistheoremname}{}
\newtheorem*{genericthm*}{\thistheoremname}
\makeatletter
\newenvironment{namedthm*}[2]
  {\renewcommand{\thistheoremname}{#1}%
   \begin{genericthm*}%
   \def\@currentlabelname{#1}
   \textrm{$($}#2\textrm{$)$}}%
  {\end{genericthm*}}
\makeatother

\theoremstyle{definition}
\newtheorem{defn}[thm]{Definition}

\theoremstyle{remark}
\newtheorem{remark}[thm]{Remark}
\newenvironment{pf}{\par\begin{trivlist}%
\item[]{\bf Proof.}\ }{\hfill $\square$ \end{trivlist}\par}
\newenvironment{apf}[1]{\par\begin{trivlist}%
\item[]{\bf Proof of #1.}\ }{\hfill $\square$ \end{trivlist}\par}

\makeatletter \@addtoreset{equation}{section} \makeatother

\newcommand{\R}{\mathbb{R}}

\newcommand{\N}{\mathbb{N}}

\DeclareMathOperator{\diam}{diam}
\DeclareMathOperator{\Hess}{Hess}

\newcommand{\supp}{\mathop{\rm supp}\nolimits}
\newcommand{\lip}{\mathrm{Lip}}

\newcommand{\Div}{\operatorname{div}}

\renewcommand{\d}{\mathrm{d}}
\newcommand{\m}{\mathfrak{m} }

\newcommand{\s}{\mathfrak{s}}

\newcommand{\loc}{{\rm loc}}

\newcommand{\<}{\langle}
\renewcommand{\>}{\rangle}

\newcommand{\sobolevpower}{\widehat{p}}
\newcommand{\uppergradpower}{{p_0}}

\title[Morrey's inequalities on CD]{On higher order Morrey's inequalities over CD\boldmath$(K,N)$-spaces}
\author[]{
Jun Kitagawa}
\address{Department of Mathematics, Michigan State University, 619 Red Cedar Road
East Lansing, MI 48824, USA}
\email{kitagawa@math.msu.edu}

\author[]{Kazuhiro Kuwae}
\address{Department of Applied Mathematics, Fukuoka University,
Fukuoka 814-0180, Japan}
\email{kuwae@fukuoka-u.ac.jp}
\date{}

\begin{document}

\begin{abstract}
In this paper, we establish a higher order Morrey's inequality in the framework of 
$\mathsf{CD}(K,N)$-spaces
for $K\in\mathbb{R}$ and $N\in\mathbb{N}$. We do so by first introducing an alternate version of second order Sobolev spaces in general metric measure spaces, which contains amply many functions even when $p>N$. These spaces contain the traditional $H^{2, p}(X)$ spaces on $\mathsf{RCD}(K, N)$-spaces.
\end{abstract}
\maketitle
\section{Introduction}\label{sec:Introduction}
Recall the classical Morrey's inequality on $\R^N$, which is the following: 
\begin{namedthm*}{Morrey's inequality}{\cite[pp.~266--269]{Evans}}
Let $U$ be a bounded open subset of $\R^N$ with $C^1$-boundary or $U=\R^N$.
Assume $u\in W^{1,p}(U)$ with $p>N$. Then there exists
$u^*\in C^{0,\gamma}(\overline{U})$, where $\gamma:=1-N/p$ with $u=u^*$ Lebesgue a.e., 
and $C=C(n,p,U)>0$ such that
\begin{align*}
\lVert u\rVert_{C^{0,\gamma}(\overline{U})}\leq C\lVert u\rVert_{W^{1,p}(U)}.
\end{align*}    
\end{namedthm*}
Moreover, there is the following higher order version:
\begin{namedthm*}{Higher order Morrey's inequality}{\cite[p.~270]{Evans}}
Let $U$ be a bounded domain in $\R^N$ with $C^1$-boundary.
Assume $u\in W^{k,p}(U)$ with $p>\frac{N}{k}$. Then 
there exists $u^*\in C^{k-\left[\frac{N}{p} \right]-1,\gamma}(\overline{U})$ where 
\begin{align*}
\gamma:=\left\{\begin{array}{lc}\left[ \frac{N}{p}\right]+1-\frac{N}{p}, & \frac{N}{p}\notin\mathbb{Z}, \\ \text{any positive number}<1, & \frac{N}{p}\in\mathbb{Z},\end{array}\right.
\end{align*}
with $u=u^*$ Lebesgue a.e., 
and there exists $C=C(k,N,p,U)>0$ such that
\begin{align*}
\lVert u\rVert_{C^{k-\left[\frac{N}{p}\right]-1,\gamma}(\overline{U})}\leq C\lVert u\rVert_{W^{k,p}(U)}.
\end{align*}
As a special case, if $u\in W^{2,p}(U)$ and $p>\frac{N}{2}$, then $u^*\in C^{1-\left[\frac{N}{p} \right],\gamma}(\overline{U})$ with $u=u^*$ Lebesgue a.e.,
and there exists $C=C(2,N,p,U)>0$ such that
\begin{align*}
\lVert u\rVert_{C^{1-\left[\frac{N}{p}\right],\gamma}(\overline{U})}\leq C\lVert u\rVert_{W^{2,p}(U)}.
\end{align*}
\end{namedthm*}
The above theorem has many applications in the theory of partial differential equations. For example, the classical Morrey's inequality ensures that 
any function in $W^{1,p}_{\rm loc}(U)$ is 
differentiable a.e.~in $U$, and its gradient coincides with the weak gradient a.e.~in $U$ (see \cite[p.~295, Theorem~5]{Evans}.
Moreover, the higher order version is used in a critical way in the theory of $L^p$-viscosity solutions (see \cite{CaffarelliCrandallKokanSwidgi}): in that setting test functions are taken in $W_{\loc}^{2,p}(U)$ where $p>N/2$, which are continuous by the higher order Morrey inequality. 

\medskip

In this paper, we establish a higher order Morrey's inequality for functions in second order Sobolev spaces in the framework of $\mathsf{CD}(K,N)$-spaces. Here, 
$\mathsf{CD}(K,N)$-spaces are certain metric measure spaces  
having a notion of Ricci curvature bound from below by $K$ and dimension bounded from above by $N$. 
The theory of 
$W^{1,p}(X)$-Sobolev spaces for $p\in\,]\,1,\,+\infty\,[$ over a metric measure space $(X,\mathsf{d},\m)$ was formulated by Ambrosio, Gigli, and Savar\'e (\cite{AGS_Sobolev}) in terms of $p$-weak upper gradients by using mass transport theory without assuming $\mathsf{CD}(K,N)$-structure.  On the other hand, Gigli~\cite{Gigli:NonSmoothDifferentialStr} established a theory of  
(non-smooth) second order calculus on $\mathsf{RCD}(K,N)$-spaces and defined the $W^{2,2}(X)$-Sobolev space over such a space $(X,\mathsf{d},\m)$ in terms of an object called the $L^2$-Hessian, where an $\mathsf{RCD}(K,N)$-space is a 
$\mathsf{CD}(K,N)$-space with Riemannian structure.  
For an $\mathsf{RCD}(K,N)$-space $(X,\mathsf{d},\m)$, there are a number of candidates for the
$W^{2,p}(X)$-Sobolev space for any $p\in\,]\,1,\,+\infty\,[$, but there may be a possibility 
that $W^{2,p}(X)$ consists of only constant functions if the underlying space $(X,\mathsf{d},\m)$ itself is too singular, which may be the case for certain Aleksandrov spaces. To avoid this trivial case, we introduce a 
new modified Sobolev space $W^{2,p}_*(X)$ containing $W^{2,p}(X)$ and prove a higher order version of Morrey's inequality for functions in $W^{2,p}_*(X)$ (see Theorem~\ref{thm: W2,p morrey}). 

\medskip

 The constitution of this paper is as follows: In 
 Section~\ref{sec:StatementMain}, we explain the basic framework and state our higher order Morrey estimate Theorem~\ref{thm: W2,p morrey}; we also include a heuristic discussion of various notions of ``second order Sobolev space.'' 
In Section~\ref{sec:1pSobolev}, we recall the framework of $(1,p)$-Sobolev spaces and $\mathsf{CD}$ spaces in terms of the theory of optimal mass transport. 
In Section~\ref{sec: W2p}, we give the definition of the ``traditional'' $(2,p)$-Sobolev spaces $W^{\hspace{0.03cm}2,p}(X)$ and $H^{\hspace{0.03cm}2,p}(X)$ in terms of the notion of $L^p$-normed $L^{\infty}$-module developed by Gigli~\cite{Gigli:NonSmoothDifferentialStr} 
for $\mathsf{RCD}(K,N)$-spaces and introduce the ``alternative'' $(2,p)$-Sobolev space $W^{\hspace{0.03cm}2,p}_*(X)$ for $p\in\,[1,\,+\infty\,[$ over any metric measure space $(X,\mathsf{d},\m)$, along with its subspace $H^{\hspace{0.03cm}2,p}_*(X)$, defined on an $\mathsf{RCD}(K, N)$-space. We endow $W^{\hspace{0.03cm}2,p}_{\ast}(X)$ with a norm and prove it yields a Banach space, and that $H^{2, p}(X)$ can be $1$-Lipschitz embedded into $W^{2, p}_*(X)$ when $(X, \mathsf{d}, \m)$ is an $\mathsf{RCD}(K, N)$-space. In Section~\ref{sec: proof of main theorem} we first give the proof of Theorem~\ref{thm: W1p sobolev} which contains local first order Sobolev and Morrey inequalities based on the theory of Haj{\l}asz and Koskela developed in \cite{HK}. 
Based on these, we
prove our second order Morrey inequality Theorem~\ref{thm: W2,p morrey}.  Finally in Section~\ref{sec: W2,p functions}, we clarify that the modified Sobolev spaces $W^{\hspace{0.03cm}2,p}_{\ast}(X)$ and
$H^{\hspace{0.03cm}2,p}_{\ast}(X)$ are sufficiently ample,
the former in general and the latter when $(X,\mathsf{d},\m)$ is an $\mathsf{RCD}(K, N)$-space with finite total measure and the
measures of balls of a fixed radius do not collapse.

\section{Statement of Main Theorem}\label{sec:StatementMain}

\subsection{Framework}\label{subsec:Frame}
Various terms and spaces will be defined precisely in Section~\ref{sec:1pSobolev}, but in particular, $W^{\hspace{0.03cm}2, p}_\ast(X)$ are the modified Sobolev spaces of higher order introduced in Definition~\ref{df:modifiedSobolev} below, defined for general $p\in \,[1,\,+\infty\,[$, and $\mathcal{UG}^p(f)$ denotes the set of $p$-weak upper gradients of a function $f$, given in Definition~\ref{def:Sobolev}. Now we can state our main theorem: 
\medskip
\begin{thm}[Higher order Morrey's inequality]\label{thm: W2,p morrey}
   Let $(X, \mathsf{d}, \m)$ be a $\mathsf{CD}(K, N)$-space with $K\in \R$ and $N\in [\,1,\, +\infty\,[$. 
    Also, suppose that $f\in W^{\hspace{0.03cm}2, p}_\ast(X)$ with $p\in [\,1,\,+\infty\,[\,\cap\, ]\,N/2, \,+\infty\,[$. Then there exists a locally H{\"o}lder continuous function $f^*$ such that $f=f^*$ $\m$-a.e. 
    More precisely, we have the following:
    \begin{enumerate}[{\rm (1)}]
     \item\label{eq:Holder1} If $p>N$, then for any $G\in \mathcal{UG}^{\,p}(f)$, $x\in X$, $r\in \,]\,0,\,R\,]$, 
     we have $G\in L^p(B_R(x))$
     and for any $y$, $z\in B_r(x)$,  
    \begin{align*}
        \lvert f^*(y)-f^*(z)\rvert \leq C_{N, K, x, q, R}\, r^{1-\frac{N}{p}}\left(\int_{B_R(x)}G^{\,p} \,\d\m\right)^{\frac{1}{p}}.
    \end{align*}
     \item\label{eq:Holder2} If $p=N$, then  for any   $q>N=p$, $G\in \mathcal{UG}^p(f)\cap W^{\hspace{0.03cm}1,p}(X)$, $x\in X$, $r\in \,]\,0,\,R\,]$, we have $G\in L^q(B_R(x))$ and for any $y$, $z\in B_r(x)$, 
    \begin{align*}
        \lvert f^*(y)-f^*(z)\rvert \leq C_{N, K, x, k, R}\, r^{1-\frac{N}{q}}\left(\int_{B_R(x)}G^{\,q} \,\d\m\right)^{\frac{1}{q}}.
    \end{align*}
     \item\label{eq:Holder3} If $N/2<p<N$, then 
     for any $G\in \mathcal{UG}^p(f)\cap W^{\hspace{0.03cm}1,p}(X)$, $x\in X$, $r\in \,]\,0,\,R\,]$, we have $G\in L^{\frac{Np}{N-p}}(B_R(x))$ and for any $y$, $z\in B_r(x)$, we have
    \begin{align*}
        \lvert f^*(y)-f^*(z)\rvert \leq C_{N, K, x, R}\, r^{2-\frac{N}{p}}\left(\int_{B_R(x)}G^{\frac{Np}{N-p}} \,\d\m\right)^{\frac{N-p}{Np}}.
    \end{align*}
    \end{enumerate}

\end{thm}

A few comments on the notion of ``$W^{2, p}$ Sobolev space'' when $p\neq 2$ are in order. When $p=2$, the theory established by Gigli via the language of $L^2$-normed $L^\infty$-modules yields an object that can rightly be called ``the Hessian'' of a function for some class of functions $W^{2, 2}(X)$ on an $\mathsf{RCD}(K,N)$-space (see~\cite[Chapters 4 and 6]{GPLecture} for further details). There are at least two natural candidates to extend this definition to spaces $W^{2, p}(X)$ for $p\neq 2$. The first is to define $\Hess f$ for a function $f\in W^{1, p}(X)$ as an element of the module $L^2((T^*)^{\otimes2}X)$ for an 
$\mathsf{RCD}(K,N)$-space $(X,\mathsf{d},\m)$ 
(as in \cite[Definition 6.2.6]{GPLecture}), then declare $f\in W^{2, p}(X)$ if the real valued function $\lvert \Hess f\rvert_{\rm HS}$ belongs to $L^p(X)$. 
However, under this approach it may not be possible to show (with the current technology) that $W^{2, p}(X)$ contains any elements besides constant functions, even in relatively nice cases, for example when $X$ is a non-collapsed $\mathsf{RCD}(K,N)$-space (as was communicated to the authors by Shouhei Honda). Indeed, in such a case, if $f\in W^{2, p}(X)$ belongs to the domain of the Laplacian (see~\eqref{eqn: Laplacian})
one finds that $\Delta f$, as the trace of $\Hess f$ $\m$-a.e. in a non-collapsed $\mathsf{RCD}(K,N)$-space $(X, \mathsf{d}, \m)$ by \cite[Theorem 1.5]{BrenaGigliHondaZhu23}, satisfies $\Delta f\in L^p(X)$. If $p>N$, by the local Morrey's inequality there exists a $\m$-version of $\lvert Df\rvert_p$ which is continuous, hence \cite[Theorem 1]{DePhilippisZimbron23} yields that $\lvert Df\rvert_p$ must be zero at any singular point of $(X, \mathsf{d}, \m)$. Since it is possible to construct non-collapsed $\mathsf{RCD}(K,N)$-spaces where the set of such points is dense (for example, \cite[p.~ 632, Examples (2)]{OtsuShioya94} yields an example with $K=0$ for any $N$; this example is even an Aleksandrov space), for such a space $W^{2, p}(X)\cap D(\Delta)$ will consist only of functions with $\lvert Df\rvert_p=0$ everywhere. By the so-called Sobolev-to-Lipschitz property (\cite[Theorem 6.2]{AGS_Riem}) valid in ${\mathsf{RCD}}(K,N)$-spaces, this implies such a function must be constant. In particular, the set ${\rm Test}^\infty(X)$ of \lq\lq test functions\rq\rq\ (see Definition~\ref{def:TestFunc}), is a subset of $D(\Delta)$ hence nonconstant elements of ${\rm Test}^\infty(X)$ cannot belong to $W^{2, p}(X)$ as above if $p>N$. This is in sharp contrast to the case $p=2$, where it is shown that ${\rm Test}^\infty(X)\subset W^{2, 2}(X)$ in \cite[Section 6.2.3]{GPLecture}.

The other natural candidate is to declare $\Hess f$ for a function $f$ to be an element of the module $L^p((T^*)^{\otimes2}X)$ satisfying \cite[Definition 6.2.6]{GPLecture} for a different, suitable class of test functions depending on the value of $p$ in the framework of $\mathsf{RCD}(K,N)$-spaces. However, the tools used to obtain the inclusion ${\rm Test}^\infty(X)\subset W^{2, 2}(X)$ in \cite[Section 6.2.3]{GPLecture} are specific to $p=2$. Additionally, it is not at all clear if a sufficiently ample class of elements of $W^{2, p}(X)$, should they be defined this way, can be found.

Another, third approach would be to declare that a function $f$ belongs to the alternate space $W^{2, p}_\ast(X)$ if $f\in W^{1, p}(X)$ and $\lvert D\lvert Df\rvert_p\rvert_p\in L^p(X)$ in the framework of general metric measure spaces. We do not show this in the present paper, but it can be seen that such a characterization is strong enough to obtain the higher order Morrey's inequality Theorem~\ref{thm: W2,p morrey}, the resulting space contains the space $W^{2, p}(X)$ when $(X,\mathsf{d},\m)$ is an 
$\mathsf{RCD}(K,N)$-space defined according to the first proposed way above. However, this approach has a major flaw in that the sum of two such functions may not satisfy the same condition: suppose $X\subset \R$ is a bounded interval containing two points $x_1$ and $x_2$, equipped with the usual Euclidean metric and Lebesgue measure, then $f_1(x)=\lvert x-x_1\rvert$ and $f_2(x)=\lvert x-x_2\rvert$ belong to $W^{1, p}(X)$ with $\lvert D\lvert Df_i\rvert_p\rvert_p\in L^p(X)$, but $\lvert D(f_1+f_2)\rvert$ contains jumps and is not weakly differentiable.

Thus, we have opted to take as an alternate space $W^{2, p}_*(X)$, the set of functions in $W^{1, p}(X)$ for which \emph{some} weak upper gradient belongs to $W^{1, p}(X)$. We show that it is possible to make this into a Banach space, and (a dense subspace of) the space $W^{2, p}(X)$ defined according to the first philosophy above can be realized as a $1$-Lipschitz embedding into our $W^{2, p}_\ast(X)$ 
when $(X,\mathsf{d},\m)$ is an 
$\mathsf{RCD}(K,N)$-space (see Propositions~\ref{prop: Banach space} and~\ref{prop: continuous embedding}). Under this definition, $W^{2, p}_*(X)$ will contain many functions as we show in Section~\ref{sec: W2,p functions}; note that even in the simplest case of $X\subset\R^N$ with the usual Euclidean metric, the space $W^{2, p}_\ast(X)$ may \emph{not} be the same as $W^{2, p}(X)$, even for $p=2$: 
if $X$ is bounded and contains the origin, the function $u(x)=\lvert x\rvert$ belongs to $W^{2, p}_\ast(X)$ as $u\in W^{1, p}(X)$ and $\lvert \nabla u(x)\rvert_p=\lvert \frac{x}{\lvert x\rvert}\rvert\equiv 1$. However, when $n=1$ the distributional second derivative is a delta measure supported at $x_0$, and for $n\geq 2$ the $(i, j)$-distributional second derivative is $\delta_{ij}\lvert x\rvert^{-1}-x_ix_j\lvert x\rvert^{-3}$, hence $u\notin W^{2, p}(X)$ for $p\geq n$. Nonetheless, the space will be strictly smaller than $W^{1, p}(X)$ in this case, see Section~\ref{sec: W2,p functions}.

We also mention some history of first order Morrey type inequalities. In the framework of complete smooth Riemannian manifolds with positive injectivity radius and bounded sectional curvature, 
Morrey's inequality was implicitly proved by Aubin~\cite[Subsection~2.23, Proof of the Sobolev embedding theorem~2.21]{Aubin}.  In the framework of $N$-dimensional Aleksandrov spaces with $N\in\mathbb{N}$ and a lower curvature bound $K\in\mathbb{R}$, Morrey's inequality was shown by Kuwae--Machigashira--Shioya~\cite[Theorem~7.3]{KMS:lap}. 
In Haj{\l}asz--Koskela~\cite[Theorem~5.1(3)]{HK}, a general Morrey's inequality is proved in the framework of 
metric measure space admitting global volume doubling condition,  \cite[(21)]{HK} and $(1,p)$-Poincar\'e inequality in terms of upper gradients. 
Also in Alonso-Luiz--Baudoin~\cite[Theorem~5.9]{LuizBaudoin:Morrey}, 
a generalization of Morrey's inequality is proved in the framework of metric measure space satisfying a weak Bakry--\'Emery-type condition, but
\cite[Theorem~5.9]{LuizBaudoin:Morrey}, or \cite[Theorem~5.1(3)]{HK} requires a global volume
doubling condition as in Coulhon~\cite[(D)]{Coulhon}. 
\\
\quad
\\
\section{\texorpdfstring{$(1,p)$}{}-Sobolev spaces}\label{sec:1pSobolev}
\subsection{Notation and setup}\label{subsec:notation}

A \emph{metric measure space} is a triple $(X,\mathsf{d},\m)$ such that 
\begin{align*}
(X,\mathsf{d}):\quad &\text{ is a complete separable metric space},\\
\m\ne0:\quad&\text{ is a non-negative and boundedly finite Borel measure}.
\end{align*}
Any metric open ball will be denoted by $B_r(x):=\{y\in X\mid \mathsf{d}(x,y)<r\}$ for $r>0$ and $x\in X$; recall we assume $\m(B_r(x))>0$ for any open ball. 
For any $f\in L^p_{\loc}(X;\m)$, its support ${\rm supp}[f]$ is defined to be the support of the measure $\lvert f\rvert\m$, i.e. 
\begin{align*}
{\rm supp}[f]:=\left\{x\in X\;\left|\; \int_{B_r(x)}\lvert f\rvert\,\d\m>0\quad\text{ for any }\quad r>0\right.\right\}.
\end{align*} 
Furthermore, since it will be generally  fixed throughout, we will omit the measure $\m$ from the notation of $L^p$ 
spaces, writing simply $L^p(X)$. We will also have use for various Sobolev spaces whose definitions also depend on the reference measure $\m$, but we will also suppress this notation. However, we will make an exception in Section~\ref{sec: W2,p functions} where we will need to keep track of the relevant reference measures.

\subsection{{\boldmath$q$}-test plans and {\boldmath$(1,p)$}-Sobolev spaces}\label{subsec:$q$-test plan}
Here we recall the definition of $(1,p)$-Sobolev spaces proposed in \cite{AGS_Sobolev} (see also the original works \cite{Shanmugalingam:Newton} and \cite{Ch:metmeas} for earlier approaches). 

Denote by $C([\,0,\,1\,], X)$ the space of continuous curves in $X$ defined on the unit interval $[\,0,\,1\,]$, with the distance $\mathsf{d}_{\infty}(\gamma,\eta):=\sup_{t\in[\,0,\,1\,]}\mathsf{d}(\gamma_t,\eta_t)$ for every $\gamma,\eta\in C([\,0,\,1\,], X)$; this turns $C([\,0,\,1\,], X)$ into a complete separable metric space.   
For $q\in \,]1,\, +\infty\,]$ satisfying $1/p+1/q=1$, we consider the set of $q$-absolutely continuous curves, 
\begin{align*}
    &AC^{q}([\,0,\,1\,],X)\\
    :&=\left\{\gamma\in C([\,0,\,1\,], X)\;\left|\; \exists g\in L^{q}([\,0,\,1\,])\ s.t.\ \mathsf{d}(\gamma_t,\gamma_s)\leq\int_s^tg(r)\,\d r,\quad s<t\text{ in }[\,0,\,1\,]\right.\right\}.
\end{align*}
Recall that for any $\gamma\in AC^{q}([\,0,\,1\,],X)$, there exists a minimal (pointwise a.e.)~function $g\in L^{q}([\,0,\,1\,])$ satisfying the above, called the {\it metric speed}, denoted by $\lvert \dot\gamma\rvert$, which is defined as 
$\lvert \dot\gamma\rvert:=\lim_{h\downarrow0}\mathsf{d}(\gamma_{t+h}, \gamma_t)/h$ when this limit exists, 
$\lvert \dot\gamma\rvert:=+\infty$ otherwise. If $\gamma\in AC^{q}([\,0,\,1\,],X)$, the above limit exists for Lebesgue a.e. $t\in[0, 1]$, and the function $\gamma\mapsto \lvert \dot \gamma_t\rvert$ is Borel on $C([0, 1], X)$ (see \cite[Theorem 1.2.5 and Remark 1.2.6]{GPLecture}).
We then define the modified energy functional $\mathsf{Ke}_q^*$
by
\begin{align*}
    C([\,0,\,1\,], X)\ni\gamma\mapsto {\mathsf{Ke}_q^*(\gamma)}:=
    \begin{cases}
        {\lVert\lvert\dot{\gamma} \rvert\rVert_{L^{q}([\,0,\,1\,])
        }}
        , &\gamma\in AC^{q}([\,0,\,1\,],X),\\
        +\infty,&\text{else}.
    \end{cases}
\end{align*}
The kinetic energy functional $\mathsf{Ke}_q$ is given by 
$\mathsf{Ke}_q(\gamma):=\mathsf{Ke}_q^*(\gamma)^{q}$ when  $q\in\,]1,\,+\infty\,[$.  Also for every $t\in[\,0,\,1\,]$, the {\it evaluation map} at time $t$ is  
$\mathsf{e}_t: C([\,0,\,1\,], X)\to X$, $\mathsf{e}_t(\gamma):=\gamma_t$ for $\gamma\in C([\,0,\,1\,], X)$. 
We easily see that $\mathsf{e}_t$ is a $1$-Lipschitz map. Below, $(\mathsf{e}_t)_{\sharp}\bm{\pi}$ denotes the pushforward of the measure $\bm{\pi}$ by the map $\mathsf{e}_t$.

\begin{defn}[{{\bf \boldmath$q$-test plan}}]\label{def:$q$-test}
Let $(X,\mathsf{d},\m)$ be a metric measure space and $q\in\,]1,+\infty]$.
A measure {\boldmath$\pi$}$\in\mathscr{P}(C([\,0,\,1\,], X))$ is said to be a $q$-test plan, provided 
\begin{enumerate}[(i)]
\item\label{item:qtest1} there exists $C>0$ so that $(\mathsf{e}_t)_{\sharp}\bm{\pi}\leq C\m$ for every $t\in[\,0,\,1\,]$;
\item\label{item:qtest2} 
\begin{align*}
\begin{cases}
        \int_{C([\,0,\,1\,], X)}\mathsf{Ke}_q(\gamma)\bm{\pi}(\d\gamma)<\infty,
        &q\in \,]1,+\infty\,[,\\
        {\rm Lip}(\cdot)\in L^{\infty}(C([\,0,\,1\,];X), 
\bm{\pi}),
&q=\infty,
    \end{cases}
\end{align*} 
\end{enumerate} 
 where 
\begin{align*}
{\rm Lip}(\gamma):=\sup_{t_1, t_2\in [0, 1], t_1\neq t_2}\frac{\mathsf{d}(\gamma_{t_1}, \gamma_{t_2})}{\lvert t_1-t_2\rvert}.
\end{align*}
\end{defn}
In the case $q=\infty$ above, it can be seen that~\eqref{item:qtest2} holds for every $\gamma$ in the support of {\boldmath$\pi$} by the lower semi continuity  of the global Lipschitz constant with respect to uniform convergence.
\begin{defn}[{{\bf \boldmath$p$-weak upper gradients and the Sobolev class $S^{\hspace{0.03cm}p}(X)$}}]\label{def:Sobolev}
Let $(X,\mathsf{d},\m)$ be a metric measure space and $p\in [\,1,\,+\infty\,[$.
For a Borel function $f$ on $X$, we say a nonnegative Borel $G$ on $X$ is a \emph{$p$-weak upper gradient of $f$} (written $G\in \mathcal{UG}^p(f)$) if 
\begin{align}
\int_{C([\,0,\,1\,], X)}\lvert f(\gamma_1)-f(\gamma_0)\rvert\text{\boldmath$\pi$}(\d\gamma)\leq \int_{C([\,0,\,1\,], X)}\int_0^1
G(\gamma_t)\lvert\dot{\gamma}_t\rvert\,\d t\,\text{\boldmath$\pi$}(\d\gamma),\quad \text{$\forall${\boldmath$\pi$}, 
$q$-test plan}. \label{eq:SpSobolev}
\end{align} 
If there exists some $G\in \mathcal{UG}^p(f)\cap L^p(X)$, we write $f\in S^{\hspace{0.03cm}p}(X)$.
We will also write  
\begin{align*}
 {\rm UG}^{\hspace{0.03cm}p}(f):=\inf\left\{\lVert G \rVert_{L^p(X)}\mid G\in \mathcal{UG}^{\hspace{0.03cm}p}(f){\cap L^p(X)} \right\}.
 \end{align*}
\end{defn}
 We note that, in contrast to the usual definition, we do not require a weak upper gradient $p$ to belong to $L^p(X)$. Also, regarding the well-definedness of Definition~\ref{def:Sobolev}: the assignment $(t,\gamma)\mapsto G(\gamma_t)\lvert\dot{\gamma}_t\rvert$ is Borel measurable (see \cite[Remark~2.1.9]{GPLecture}) when $G$ is Borel, and the right hand side of 
\eqref{eq:SpSobolev} is finite for $G\in L^p(X)$ (see \cite[(2.5)]{GigliNobili}). Thus, for $G\in \mathcal{UG}^p(f)\cap L^p(X)$, the right hand side of \eqref{eq:SpSobolev} is finite, 
and the assignment 
$L^p(X)\ni G\mapsto  \int_{C([\,0,\,1\,], X)}\int_0^1 G(\gamma_t)\lvert\dot{\gamma}_t\rvert\d t\text{\boldmath$\pi$}(\d\gamma)$ is continuous. Clearly convex combinations of $p$-weak upper gradients are also $p$-weak upper gradients, which shows that the set $\mathcal{UG}^p(f)\cap L^p(X)$ for a given Borel function $f$ is a closed, convex subset of $L^p(X)$. Thus if $f\in S^p(X)$, the \emph{minimal $p$-weak upper gradient of $f$}, denoted by $\lvert Df\rvert_p$ is the element of 
minimal $L^p$-norm in the set $\mathcal{UG}^p(f)\cap L^p(X)$. Various properties and existence of minimal $p$-weak upper gradients are detailed in \cite[Chapter 2.1]{GPLecture} when $p=2$ and when $p>1$ in 
\cite[Definition~2.4]{GigliNobili}.
However, an explicit treatment when $p=1$ seems to be missing from the literature, so for the sake of completeness we state here the following.
\begin{prop}\label{prop:minimalWeakUpper}
For $f\in S^{\hspace{0.03cm}1}(X)$, there exists a unique minimizer $G_f$ 
of ${\rm UG}^{\hspace{0.03cm}1}(f)$,
such that $G_f\leq G$ $\m$-a.e.~on $X$ for any $G\in \mathcal{UG}^{\hspace{0.03cm}1}(f)\cap L^1(X)$.  We define $\lvert Df\rvert_1:=G_f$ and call it the minimal $1$-weak upper gradient of $f\in S^{\hspace{0.03cm}1}(X)$.  
\end{prop}
The tools needed to prove Proposition~\ref{prop:minimalWeakUpper} are actually mostly already developed in \cite[Chapter 2.1]{GPLecture}. The key tool is the following lattice property.
\begin{lem}\label{lem:lattice}
    If $f\in S^{\hspace{0.03cm}1}(X)$ and $G_1$, $G_2\in \mathcal{UG}^{\hspace{0.03cm}1}(f)\cap L^1(X)$, then $G_1\land G_2\in \mathcal{UG}^{\hspace{0.03cm}1}(f)\cap L^1(X)$. 
\end{lem}
\begin{proof}[{\bf Proof}]
It can be seen that the characterization in \cite[Theorem~2.1.21]{GPLecture} for $p=2$ carries over in the same manner for the case $p=1$ (in particular, it does not rely on existence of a minimal $p$-weak upper gradient). Then the proof of \cite[Proposition~2.1.13]{GPLecture} similarly carries over for $p=1$.
\end{proof}
\begin{lem}\label{lem:stability}
    Suppose $(f_n)\subset S^{\hspace{0.03cm}1}(X)$ converges pointwise $\m$-a.e.~to some Borel function $f: X\to\mathbb{R}$ as $n\to\infty$. Also let $G_n\in \mathcal{UG}^{\hspace{0.03cm}1}(f_n)\cap L^1(X)$ for each $n\in\mathbb{N}$ be such that $G_n\to G$ in $L^1(X)$ for some $G\in L^1(X)$. Then 
    $f\in S^{\hspace{0.03cm}1}(X)$ and $G\in\mathcal{UG}^{\hspace{0.03cm}1}(f)\cap L^1(X)$. 
\end{lem}
\begin{proof}[{\bf Proof}]
As above, the proof is quite similar to the proof of \cite[Proposition~2.1.13]{GPLecture} hence we omit it. 
\end{proof}
 \begin{proof}[\bf Proof of Proposition~\ref{prop:minimalWeakUpper}]
 Let $\{G_n\}\subset \mathcal{UG}^{\hspace{0.03cm}1}(f)\cap L^1(X)$ be a minimizing sequence for ${\rm UG}^{\hspace{0.03cm}1}(f)$, then by Lemma~\ref{lem:lattice}, we may assume $G_n\geq G_{n+1}$ $\m$-a.e. for any $n\in\mathbb{N}$. Let $G_f:=\inf_{n\in\mathbb{N}}G_n$, then $G_n\to G_f$ as $n\to\infty$ 
 $\m$-a.e.~and $G_n\to G_f$ in $L^1(X)$ by Lebesgue's dominated convergence theorem. By applying Lemma~\ref{lem:stability}, we have $G_f\in \mathcal{UG}^{\hspace{0.03cm}1}(f)\cap L^1(X)$. From this, 
 $G_f$ is a minimizer of ${\rm UG}^{\hspace{0.03cm}1}(f)$, i.e., $\lVert G_f \rVert_{L^1}={\rm UG}^{\hspace{0.03cm}1}(f)$. 
 Next we prove that 
 $G_f\leq G$ for any $G\in\mathcal{UG}^{\hspace{0.03cm}1}(f)\cap L^1(X)$. 
 By contradiction, suppose  
 there exists some $G\in \mathcal{UG}^{\hspace{0.03cm}1}(f)\cap L^1(X)$ such that $\m(\{G<G_f\})>0$. In particular, the function $G\land G_f$, which has an $L^1(X)$-norm that is strictly smaller than $\lVert G_f \rVert_{L^1}$, is a $1$-weak upper gradient of $f$ by Lemma~\ref{lem:lattice}, a contradiction. Finally, if both $G_1$ and $G_2$ are 
 minimizers of ${\rm UG}^{\hspace{0.03cm}1}(f)$, then $G_1\leq G_2$ and $G_2\leq G_1$ hold $\m$-a.e. by the previous argument, hence $G_1=G_2$ $\m$-a.e. 
 \end{proof}

\begin{defn}[Sobolev space $W^{\hspace{0.03cm}1,p}(X)$]\label{def:W1pSobolev}
Let $(X,\mathsf{d},\m)$ be a metric measure space and $p\in [\,1,\,+\infty\,[$. The Sobolev space,  
denoted by $W^{\hspace{0.03cm}1,p}(X)$, is $L^p(X)\cap S^{\hspace{0.03cm}p}(X)$ as a set, equipped with the norm 
\begin{align*}
\lVert f\lVert_{W^{\hspace{0.03cm}1,p}(X)}:=\left(\lVert f\lVert^p_{L^p(X)}+\| \lvert Df\rvert_p\|^p_{L^p(X)} \right)^{\frac{1}{p}} ,\quad f\in  W^{\hspace{0.03cm}1,p}(X). 
\end{align*}
We will also write $W^{\hspace{0.03cm}1,p}(X)_{\rm bs}:=\{f\in W^{\hspace{0.03cm}1,p}(X)\mid {\rm supp}[f]\text{ is bounded}\}$.
\end{defn}
It is a standard fact that $(W^{\hspace{0.03cm}1,p}(X),\lVert \cdot \lVert_{W^{\hspace{0.03cm}1,p}(X)})$ is a Banach space for any $p\in [\,1,\, +\infty\,[$ in view of the 
same proof as in 
 \cite[Theorem~2.1.17]{GPLecture} for $p=2$ by applying Lemma~\ref{lem:stability}. 
It should be noted that $W^{\hspace{0.03cm}1,p}(X)\cap L^{\infty}(X)$ is an algebra. 
It is in general false 
that $(W^{\hspace{0.03cm}1,p}(X),\lVert \cdot \lVert_{W^{\hspace{0.03cm}1,p}(X)})$ is reflexive and 
$p=2$ does not imply that $(W^{\hspace{0.03cm}1,2}(X),\lVert \cdot \lVert_{W^{\hspace{0.03cm}1,2}(X)})$ is a Hilbert space. 
In the case it is a Hilbert space, we say that $(X,\mathsf{d},\m)$ is {\it infinitesimally Hilbertian} (see 
\cite{Gigli:OntheDifferentialStr}). Equivalently, we call $(X,\mathsf{d},\m)$  infinitesimally Hilbertian 
provided the following {\it parallelogram identity} holds:
\begin{align}
2\lvert Df\rvert_2^2+2\lvert Dg\rvert_2^2=\lvert D(f+g)\rvert_2^2+\lvert D(f-g)\rvert_2^2, \quad \m\text{-a.e.} \quad \forall f,g\in W^{\hspace{0.03cm}1,2}(X);\label{eq:parallelogram}
\end{align}  
 this allows for a bilinear form $\<D\cdot ,D\cdot\>:W^{\hspace{0.03cm}1,2}(X)\times W^{\hspace{0.03cm}1,2}(X)\to L^1(X)$ defined by 
\begin{align*}
\<Df,Dg\>:=\frac14\lvert D(f+g)\rvert_2^2-\frac14\lvert D(f-g)\rvert_2^2,\quad f,g\in W^{\hspace{0.03cm}1,2}(X).
\end{align*}
Moreover, under~\eqref{eq:parallelogram}
the bilinear form $(\mathscr{E},D(\mathscr{E}))$ defined by 
\begin{align*}
D(\mathscr{E}):=W^{\hspace{0.03cm}1,2}(X),\quad \mathscr{E}(f,g):=\int_X\<Df,Dg\>\d \m
\end{align*}
is a strongly local Dirichlet form on $L^2(X)$ by \cite[Theorem~2.1.28]{GPLecture}.

Let $\Delta$ be the 
$L^2$-generator associated with $(\mathscr{E}, D(\mathscr{E}))$ defined by 
\begin{align}
\begin{split}
    D(\Delta):&=\{u\in D(\mathscr{E})\mid \text{ there exists }w\in L^2(X)\text{ such that }\\
&\hspace{3cm}\mathscr{E}(u,v)=-\int_X wv\,\d\m\quad\text{ for any }v\in D(\mathscr{E})\},\\
\Delta u:&=w \text{ for }u\in D(\Delta)\text{ and }w\in L^2(X)\text{ satisfying the relation above, }
\end{split}\label{eqn: Laplacian}
\end{align}
which is called the \emph{Laplacian} or \emph{$L^2$-Laplacian} associated with $W^{\hspace{0.03cm}1,2}(X)$. We also denote the \emph{measure valued Laplacian} by $\bm{\Delta}$ where
\begin{align*}
\begin{split}
    D(\bm{\Delta}):&=\{u\in S^p(X)\text{ some }p>1\mid \text{ there exists a Radon measure }\\
&\hspace{2cm}\mu\text{ such that }\mathscr{E}(u,v)=-\int_X v\,\d\mu\,\quad\text{ for any }v\in \lip_{\rm bs}(X)\},\\
\bm{\Delta} u:&=\mu \text{ for }u\in D(\bm{\Delta})\text{ and a Radon measure }\mu\in L^2(X)\text{ satisfying the relation above.}
\end{split}
\end{align*}

\begin{remark}
The above definition for $W^{\hspace{0.03cm}1,p}(X)$ is based on optimal transport theory. 
There are various way to define $(1,p)$-Sobolev spaces by, for example, Cheeger or Shanmugalingam; for general $p\in\,]1,\,+\infty\,[$, they are equivalent to each other (see \cite[Remark~2.2.27 and Theorem~2.2.28]{GPLecture} for the case $p=2$ and 
\cite[Theorem~7.4]{AGS_Sobolev} for general $p\in\,]1,\,+\infty\,[$).   
\end{remark}

We also recall the following $p$-Clarkson inequality, and some consequences regarding uniqueness of minimizers of certain functionals.
\begin{prop}[{cf.~\cite[Theorem~1.7]{KwSobolev}}]
Let $(X,\mathsf{d},\m)$ satisfy 
the infinitesimally Hilbertian condition.
Fix $p\in\,]\,1,\,+\infty\,[$.
Then $(W^{1,p}(X),\lVert \cdot \lVert_{W^{1,p}(X)})$ satisfies the $p$-Clarkson inequality in the following sense: for any $f$, $g\in W^{1, p}(X)$,
\begin{align*}
&p\in[\,2,\,+\infty\,[\Rightarrow 
\left \lVert\frac{f+g}{2} \right \lVert_{W^{1,p}(X)}^p+
\left \lVert\frac{f-g}{2} \right \lVert_{W^{1,p}(X)}^p\leq
\frac12
\left \lVert f \right \lVert_{W^{1,p}(X)}^p
+
\frac12
\left \lVert g \right \lVert_{W^{1,p}(X)}^p,
\\
&p\in\,]\,1,\,2\,[\Rightarrow
\left \lVert\frac{f+g}{2} \right \lVert_{W^{1,p}(X)}^q+
\left \lVert\frac{f-g}{2} \right \lVert_{W^{1,p}(X)}^q\leq
\left(\frac12
\left \lVert f \right \lVert_{W^{1,p}(X)}^p
+
\frac12
\left \lVert g \right \lVert_{W^{1,p}(X)}^p\right)^{\frac{q}{p}},
\end{align*}
where $q=p/(p-1)$ is the conjugate exponent of $p$.
\end{prop}
\begin{cor}\label{cor:minimizer}
    Let $(X,\mathsf{d},\m)$ satisfy 
the infinitesimally Hilbertian condition and suppose $p\in\,]\,1,\,+\infty\,[$. 
For any closed, convex subset $\mathscr{C}$ of 
$(W^{1,p}(X),\lVert \cdot \lVert_{W^{1,p}(X)})$, there exists a unique $f_{\mathscr{C}}\in \mathscr{C}$ such that 
\begin{align*}
\lVert f_{\mathscr{C}}\lVert_{W^{1,p}(X)}^p=\inf\{\
\lVert f\lVert_{W^{1,p}(X)}^p\mid f\in\mathscr{C}
 \}.
\end{align*}
\end{cor}
\begin{proof}[\bf Proof]
 Let $\{f_n\}\subset\mathscr{C}$ be a minimizing sequence for 
 \begin{align*}
 \alpha:= \inf\{\
\lVert f\lVert_{W^{1,p}(X)}^p\mid f\in\mathscr{C}
\}.   
 \end{align*} 
 Note that $\frac{f_n+f_m}{2}\in\mathscr{C}$.
 When $p\in [\,2,\,+\infty\,[$, 
 \begin{align*}
 \left \lVert\frac{f_n-f_m}{2}\right \lVert_{W^{1,p}(X)}^p&\leq\frac12(\lVert f_n \lVert_{W^{1,p}(X)}^p+
\lVert f_m \lVert_{W^{1,p}(X)}^p)-\alpha\\
 &\to\frac12(\alpha+\alpha)-\alpha=0\quad\text{ as }\quad n,m\to\infty.
 \end{align*}
 When $p\in \,]\,1,\,2\,[$, 
 \begin{align*}
 \left \lVert\frac{f_n-f_m}{2}\right \lVert_{W^{1,p}(X)}^q&\leq\left(\frac12\lVert f_n \lVert_{W^{1,p}(X)}^p+\frac12
\lVert f_m \lVert_{W^{1,p}(X)}^p\right)^{\frac{q}{p}}-\alpha^{\frac{q}{p}}\\
&\to\left(\frac12(\alpha+\alpha)\right)^{\frac{q}{p}}-\alpha^{\frac{q}{p}}=0\quad\text{ as }\quad n,m\to\infty.
 \end{align*}
 In any case, $\{f_n\}$ forms a $W^{1,p}(X)$-Cauchy sequence in $\mathscr{C}$. 
 Since $\mathscr{C}$ is closed in $W^{1,p}(X)$, there exists an $f_{\mathscr{C}}\in\mathscr{C}$ such that 
 $f_n\to f_{\mathscr{C}}$ in $W^{1,p}(X)$ as $n\to\infty$. This implies the existence of minimizer. The proof of uniqueness can be done similarly. 
\end{proof}
\subsection{CD-spaces}\label{sec:CD}
In this subsection, we give the precise definition of $\mathsf{CD}$-spaces. 
\begin{defn}\label{def:skappasigmakappa}
{\rm For $\kappa\in\R$ and $\theta\geq0$ we define 
\begin{align}
\s_{\kappa}(\theta)
:=\left\{\begin{array}{cc}\frac{\sin \sqrt{\kappa}\theta}{\sqrt{\kappa}}, & \kappa>0, \\ 
\theta, & \kappa=0, \\ 
\frac{\sinh \sqrt{-\kappa}\theta}{\sqrt{-\kappa}}, & \kappa<0.\end{array}\right.\label{eq:Jacobi}
\end{align}
Moreover, for $t\in\,[\,0,\,1\,]$ we set
\begin{align*}
\sigma_{\kappa}^{(t)}(\theta):=\left\{\begin{array}{cl}\frac{\s_{\kappa}(t\theta)}{\s_{\kappa}(\theta)}, & \kappa\theta^2\ne0\text{ and }\kappa\theta^2<\pi^2, \\ 
t, & \kappa\theta^2=0, \\ +\infty, & \kappa\theta^2\geq\pi^2\end{array}\right.
\end{align*}
and 
\begin{align*}
\tau_{K,N}^{(t)}(\theta):=t^{1/N}\sigma_{K/(N-1)}^{(t)}(\theta)^{1-\frac{1}{N}}.
\end{align*}
}
\end{defn}
Denote by $\mathscr{P}_2(X,\mathsf{d})$ the set of all Borel probability measures $\mu$ on $(X, \mathsf{d})$ satisfying $\int_X\mathsf{d}(x_0,x)^2\d\mu( x)<\infty$ for some, hence all $x_0\in X$. The $L^2$--Monge--Kantorovich distance between $\mu_0$ and $\mu_1\in\mathscr{P}_2(X,\mathsf{d})$ is defined by 
\begin{align}
{\rm M\!\!K}_2(\mu_0,\mu_1)^2:=\inf\left\{\left.\int_X\mathsf{d}(x,y)^2\d\pi(x, y)\,\right|\,\pi\in\mathscr{P}(X\times X),\;(p_i)_{\sharp}\pi=\mu_i \;(i=0,1) 
\right\},\label{eq:2MK}
\end{align}
where $p_1$, $p_2$ are the projections defined by $p_1(x,y):=x$, $p_2(x,y):=y$. 
Let us denote by ${\rm Geo}(X):=\{\gamma: [\,0,\,1\,]\to X\mid \gamma\text{ is a constant speed geodesic}\}$  the space of constant speed geodesics in $X$ equipped with the topology of uniform convergence. 
A \emph{dynamic optimal coupling} between $\mu_0$, and $\mu_1\in\mathscr{P}_2(X,\mathsf{d})$ is a probability measure $\bm{\pi}\in\mathscr{P}({\rm Geo}(X))$ such that $(\mathsf{e}_0,\mathsf{e}_1)_{\sharp}\bm{\pi}$ is an optimal coupling of $\mu_0$, $\mu_1$, i.e., a minimizer of \eqref{eq:2MK} (recall that $\mathsf{e}_t$ is the evaluation map at time $t$ defined earlier). 
The curve $(\mu_t)_{t\in\,[\,0,\,1\,]}$ with $\mu_t=(\mathsf{e}_t)_{\sharp}\bm{\pi}$ is then a geodesic in $\mathscr{P}_2(X,\mathsf{d})$ connecting $\mu_0$ to $\mu_1$. Moreover, by work of Sturm~\cite[Lemma~2.11]{Sturm06a}, for each geodesic $\Gamma:[\,0,\,1\,]\to\mathscr{P}_2(X,\mathsf{d})$, there exists a probability measure $\bm{\pi}$ on ${\rm Geo}(X)$ such that $\Gamma_t=(\mathsf{e}_t)_{\sharp}\bm{\pi}$ for all $t\in\,[\,0,\,1\,]$. 

Given a measure $\mu\in\mathscr{P}_2(X,\mathsf{d})$ we define 
its relative entropy by ${\rm Ent}(\mu):=\int_X\rho\log\rho \d\m$ if $\mu=\rho\m$ is absolutely continuous with respect to $\m$ and $(\rho\log\rho)^+$ is $\m$-integrable. Otherwise we set ${\rm Ent}(\mu):=+\infty$. The set of probability measures with finite entropy will be denoted by $\mathscr{P}_2^*(X,\mathsf{d},\m)$. 

\begin{defn}\label{def:CDKinfty}
{\rm Fix $K\in\R$. 
we say that a metric measure space $(X,\mathsf{d},\m)$ satisfies the \emph{curvature-dimension condition} $\mathsf{CD}(K,\infty)$ if and only if for each pair $\mu_0$, $\mu_1\in\mathscr{P}_2^*(X,\mathsf{d},\m)$ there exists a constant speed geodesic $(\mu_t)_{t\in\,[\,0,\,1\,]}\in \mathscr{P}_2^*(X,\mathsf{d},\m)$ connecting $\mu_0$ to $\mu_1$ such  that for all $t\in\,[\,0,\,1\,]$: 
\begin{align}
{\rm Ent}(\mu_t)\leq (1-t){\rm Ent}(\mu_0)+t{\rm Ent}(\mu_1)-\frac{K}{2}t(1-t){\rm M\!\!K}_2^2(\mu_0,\mu_1).
\end{align}
}
\end{defn}


Denote by $\mathscr{P}_{\infty}(X,\mathsf{d},\m)$ the set of measures in 
$\mathscr{P}_2(X,\mathsf{d},\m)$ with bounded support.
\begin{defn}
{\rm A metric measure space $(X,\mathsf{d},\m)$ is said to satisfy the 
{\it curvature-dimension condition} $\mathsf{CD}(K,N)$ if and only if for each pair $\mu_0=\rho_0\m$ and $\mu_1=\rho_1\m\in\mathscr{P}_{\infty}(X,\mathsf{d},\m)$
there exists an optimal coupling $\pi$ between them and a geodesic $(\mu_t)_{t\in[\,0,\,1\,]}$ 
 connecting them such that for all $t\in[\,0,\,1\,]$ and $N'\geq N$: 
\begin{align}
\int_X\rho_t^{-\frac{1}{N'}}
\d\mu_t\geq &\int_{X\times X}\left[ 
\tau_{K,N'}^{(1-t)}(\mathsf{d}(x_0,x_1))\rho_0(x_0)^{-\frac{1}{N'}}\right.\notag\\
&\quad \left. +\tau_{K,N'}^{(t)}(\mathsf{d}(x_0,x_1))\rho_1(x_1)^{-\frac{1}{N'}}
\right]\d\pi(x_0, x_1).\label{eq:CurvDim}
\end{align}
If \eqref{eq:CurvDim} holds for \emph{any} geodesic 
$(\mu_t)_{t\in\,[\,0,\,1\,]}\in \mathscr{P}_{\infty}(X,\mathsf{d},\m)$ we say that 
$(X,\mathsf{d},\m)$ is a {\it strong} $\mathsf{CD}(K,N)$-space. 

}
\end{defn}


\begin{remark}\label{rem:CD}
{\rm \begin{enumerate}[{\rm(1)}]
\item\label{item: entropic and reduced CD} There exist two related conditions, the \emph{reduced curvature-dimension condition} introduced in~\cite{BacherSturm} and the \emph{entropic curvature-dimension condition} introduced in~\cite{EKS} (along with their \emph{strong} versions); spaces satisfying these conditions are written as $\mathsf{CD}^*(K,N)$ and $\mathsf{CD}^e(K,N)$-spaces respectively. It is known that a $\mathsf{CD}(K,N)$-space is a $\mathsf{CD}^*(K,N)$-space, and a $\mathsf{CD}^*(K,N)$-space is a $\mathsf{CD}(K^*,N)$-space with $K^*:=(N-1)K/N$. In particular, the $\mathsf{CD}(0,N)$-condition coincides with the $\mathsf{CD}^*(0,N)$-condition (see \cite[Proposition~2.5 and Remark~2.4(ii)]{BacherSturm}). Also, a $\mathsf{CD}(K,N)$-space is a $\mathsf{CD}(K,\infty)$-space provided $\m(X)<\infty$ (see \cite[Proposition~1.6(ii)]{Sturm06a}).
Both the $\mathsf{CD}^e(K,N)$-condition and the $\mathsf{CD}^*(K,N)$-condition imply
the $\mathsf{CD}(K,\infty)$-condition (see \cite[Remark~3.19]{EKS}).
\item\label{item: nonbranching CD} If $(X,\mathsf{d},\m)$ is essentially non-branching 
(see \cite[Definition~3.10]{EKS}), then the 
$\mathsf{CD}^e(K,N)$-condition is equivalent to the $\mathsf{CD}^*(K,N)$-condition (see 
\cite[Theorem~3.12]{EKS}).  Moreover, the strong  
$\mathsf{CD}^e(K,N)$-condition is equivalent to the strong $\mathsf{CD}^*(K,N)$-condition without assuming the essentially non-branching condition (see \cite[Corollary~3.13]{EKS}). 
\item\label{item: bishop gromov} If $(X,\mathsf{d},\m)$ is a $\mathsf{CD}(K,N)$-space with 
$N\in \,[\,1,\,+\infty\,[$, it enjoys the Bishop--Gromov inequality (see \cite[Theorem~3.2]{Sturm06a}): 
Let $\kappa:=K/(N-1)$ if $N>1$ and $\kappa:=0$ if $N=1$. We set $\omega_N:=\frac{\pi^{N/2}}{\int_0^{\infty}t^{N/2}e^{-t}\d t}$ (volume of the unit ball in $\R^N$ provided $N\in\mathbb{N}$) and 
$V_{\kappa}(r):=\omega_N\int_0^r\s_{\kappa}^{N-1}(t)\d t$ where $\s_{\kappa}$ is given by \eqref{eq:Jacobi}.  
 Then 
\begin{align}
\frac{\m(B_R(x))}{V_{\kappa}(R)}\leq \frac{\m(B_r(x))}{V_{\kappa}(r)},\qquad x\in X\label{eq:BG} 
\end{align}
holds for $0<r<R$ in the case $K\leq0$ and holds for $0<r<R\leq\pi\sqrt{(N-1)/K}$ in the case $K>0$. 
\item 
In the setting of~\eqref{item: bishop gromov} above, we see $\m$ satisfies the local uniform volume doubling property, i.e., for every $R>0$, there exists $C_R>0$ such that 
$\m(B_{2r}(x))\leq C_R\m(B_r(x))$ for all $x\in X$ and $r\in\,]\,0,\,R\,[$. 
In turn this implies that every closed ball is totally bounded. 
Consequently, $(X,\mathsf{d})$ is a proper metric space and its Hausdorff dimension is at most $N$ (see \cite[Corollary~2.6]{Ohta}). 
\end{enumerate}

}
\end{remark}

\subsection{RCD-spaces}\label{sec:RCD}
In this subsection, we recall the notion of $\mathsf{RCD}$-spaces. The definition of $\mathsf{RCD}(K, \infty)$-space was introduced in \cite{AGS_Riem, AGS_BakryEmery}, while $\mathsf{RCD}(K, N)$-spaces for $N<\infty$ were first introduced in~\cite{Gigli26} (first posted to arXiv in 2013).

\begin{defn}[{{\bf {\sf RCD}-spaces}}]\label{def:RCD}
{\rm A metric measure space $(X,\mathsf{d},\m)$ is said to be an {\it $\mathsf{RCD}(K,N)$-space} (resp.~{\it $\mathsf{RCD}^*(K,N)$-space}, {\it $\mathsf{RCD}^e(K,\infty)$-space}) if it is a {\it $\mathsf{CD}(K,N)$-space} (resp.~{\it $\mathsf{CD}^*(K,N)$-space}, {\it $\mathsf{CD}^e(K,N)$-space}) satisfying the infinitesimally Hilbertian condition (i.e., for which~\eqref{eq:parallelogram} holds).  
}
\end{defn}

\begin{remark}\label{rem: RCD prop}
\begin{enumerate}[{\rm(1)}]

\item\label{item:RCDe=RCD*} Any strong $\mathsf{CD}(K,N)$-space, consequently  
 $\mathsf{RCD}(K,\infty)$-space is essentially non-branching by 
 \cite[Theorem~1.1 and Corollary~1.2]{RajalaSturm14}. 
 Thus by the discussion in Remark~\ref{rem:CD}~\eqref{item: entropic and reduced CD} and~\eqref{item: nonbranching CD}  we have the implications
 \begin{align*}
     \mathsf{RCD}(K, N)
     \implies \mathsf{RCD}^*(K, N)
     \iff \mathsf{RCD}^e(K, N)
     \implies \mathsf{RCD}(K^*, N),
 \end{align*}
 where again $K^*=(N-1)K/N$. By~\cite[Corollary 13.7]{Cav-Mil} and~\cite[Corollary 1.2]{ZLi}, it is known that in fact 
 \begin{align*}
     \mathsf{RCD}(K, N)
     \iff \mathsf{RCD}^*(K, N)
     \iff \mathsf{RCD}^e(K, N).
 \end{align*}
\item If $(X,\mathsf{d},\m)$ is an $\mathsf{RCD}(K,\infty)$-space, then it is known that 
whenever $p$, $q\in \,]1,\,+\infty\,[$ and $f\in W^{1, p}(X)\cap W^{1, q}(X)$, it holds that  
$\lvert Df\rvert_q=\lvert Df\rvert_p$ $\m$-a.e. (see \cite[Proposition~3.3]{GigliHan}). This equivalence is not yet known in the case $p=1$, however we will not vary the value of $p$, thus for the remainder of the paper we will suppress this particular subscript in our notation.
\end{enumerate}
\end{remark}

\section{{$(2,p)$-Sobolev spaces}}\label{sec: W2p}
Throughout this section, we fix a metric measure space 
$(X,\mathsf{d},\m)$. We will define the notions of the $W^{2, p}(X)$ and $W^{2, p}_\ast(X)$ spaces and prove some relations between them. Also for later use we fix the following notation.
\begin{defn}
    For $f: X\to \R$ the \emph{Lipschitz constant of $f$} (possibly infinite) is
    \begin{align*}
{\rm Lip}(f):=\sup_{x_1, x_2\in X, x_1\neq x_2}\frac{\lvert f(x_1)-f(x_2)\rvert}{\mathsf{d}(x_1, x_2)}.
\end{align*}
Then the space of \emph{Lipschitz functions on $X$} is
\begin{align*}
    \lip(X):=\{f: X\to \R\mid \lip(f)<\infty\},
\end{align*}
and $\lip(X)_{\rm bs}$ will denote the set of functions in $\lip(X)$ that are identically zero outside a bounded set.
\end{defn}
This is an abuse of notation, but the difference between $\lip(\gamma)$ for a curve $\gamma$ will be apparent from context.

\subsection{{Traditional \boldmath$(2,p)$}-Sobolev spaces}\label{subsec:(2,p)Sobolev}
In this subsection we give a definition of the space $W^{2, p}(X)$ based on the notion of Hessian as in the case of $p=2$, defined via the language of normed modules.

We start by recalling the algebra ${\rm Test}(X)$ of \emph{test functions} on $(X,\mathsf{d},\m)$. These represent the \lq\lq smoothest possible objects\rq\rq\, on $X$ and are used (instead of $C_c^{\infty}(X)$ in the smooth setting) to define several differential operators via suitable integration-by-parts formulae.

\begin{defn}\label{def:TestFunc}
Let $(X,\mathsf{d},\m)$ be an $\mathsf{RCD}(K,\infty)$-space. Various spaces of {\it test functions} are defined by 
\begin{align*}
{\rm Test}(X):&=\{f\in D(\Delta)\cap L^{\infty}(X)\mid |Df|\in L^{\infty}(X),\Delta f\in W^{\hspace{0.03cm}1, 2}(X)\},\\
{\rm Test}^{\infty}(X):&=\{f\in {\rm Test}(X)\mid \Delta f\in L^{\infty}(X)\},\\
{\rm Test}^{\infty}(X)_{\rm bs}:&=\{f\in {\rm Test}^{\infty}(X)\mid {\rm supp}[f] \text{ is bounded}\}.
\end{align*}
\end{defn}
By definition of $\mathsf{RCD}(K,\infty)$-space, 
${\rm Test}^{\infty}(X)_{\rm bs}\subset{\rm Test}^{\infty}(X)\subset {\rm Test}(X)\subset {\rm Lip}(X)$, 
so we can rewrite the above as 
\begin{align*}
{\rm Test}(X)&=\{f\in D(\Delta)\cap L^{\infty}(X)\cap {\rm Lip}(X)\mid \Delta f\in W^{\hspace{0.03cm}1, 2}(X)\},\\
{\rm Test}^{\infty}(X)&=\{f\in D(\Delta)\cap L^{\infty}(X)\cap {\rm Lip}(X)\mid \Delta f\in W^{\hspace{0.03cm}1, 2}(X)\cap L^{\infty}(X)\},
\end{align*}
respectively. It is shown in 
\cite[Lemma~3.2]{Sav14} 
(resp.~\cite[Theorem~6.1.11]{GPLecture}) that 
${\rm Test}(X)$ (resp.~${\rm Test}^{\infty}(X)$) is an algebra, hence so is ${\rm Test}^{\infty}(X)_{\rm bs}$. 
Since ${\rm Lip}(X)_{\rm bs}\subset W^{\hspace{0.03cm}1,p}(X)$ for $p\in [\,1,\,+\infty\,[$, we see
\begin{align*}
{\rm Test}^{\infty}(X)_{\rm bs}\subset W^{\hspace{0.03cm}1,p}(X).
\end{align*}
Further necessary properties of the above spaces are shown in Appendix~\ref{sec: test functions}.

Denote by  $L^2(T^*\!X)$ (resp. $L^2(TX)$) 
the cotangent module 
(resp.~tangent module) defined in 
\cite[Definition~2.2.1]{Gigli:NonSmoothDifferentialStr} (resp.~\cite[Definition~2.3.1]{Gigli:NonSmoothDifferentialStr}). 
By way of \cite[Section~1.3.1]{Gigli:NonSmoothDifferentialStr}, we can define the $L^0$-modules 
$L^0(T^*\!X)$ and $L^0(TX)$ associated to $L^2(T^*\!X)$ and $L^2(TX)$, respectively, i.e.,
\begin{align*}
L^0(T^*\!X):=L^2(T^*\!X)^0,\quad  L^0(TX):=L^2(TX)^0.
\end{align*}
The characterization of Cauchy sequences in these spaces grants that 
the point-wise norms $\lvert \cdot\rvert:L^2(T^*\!X)\to L^2(X)$ and $\lvert \cdot\rvert:L^2(TX)\to L^2(X)$ as well as the musical isomorphisms $\flat:L^2(TX)\to L^2(T^*\!X)$ and $\sharp:L^2(T^*\!X)\to L^2(TX)$ uniquely extend to (non-relabeled) continuous map $\flat:L^0(TX)\to L^0(T^*\!X)$ and $\sharp:L^0(T^*\!X)\to L^0(TX)$. Then 
denote the two-fold tensor products of $L^2(T^*\!X)$ and $L^2(TX)$, respectively, in the sense of
\cite[Definition~1.5.1]{Gigli:NonSmoothDifferentialStr}
by 
\begin{align*}
L^2((T^*)^{\otimes2}X):=L^2(T^*\!X)\otimes L^2(T^*X),\quad 
L^2((T)^{\otimes2}X):=L^2(TX)\otimes L^2(TX).
\end{align*}
By \cite[(3.2.6)]{Gigli:NonSmoothDifferentialStr}, both are separable $L^2$-normed Hilbert modules. They are point-wise isometrically module isomorphic: the respective pairing is initially defined by 
\begin{align*}
(\omega_1\otimes\omega_2)(X_1\otimes X_2):=\omega_1(X_1)\omega_2(X_2)\quad\m\text{-a.e.}
\end{align*}
for $\omega_1,\omega_2\in L^2(T^*X)\cap L^{\infty}(T^*X)$ and $X_1,X_2\in L^2(TX)\cap L^{\infty}(TX)$, 
and is extended by linearity and continuity to $L^2((T^*)^{\otimes2}X)$ and $L^2((T)^{\otimes2}X)$, respectively. By a slight abuse of notation, this pairing, with \cite[Theorem~1.2.24]{Gigli:NonSmoothDifferentialStr}, induces the musical isomorphisms $\flat:L^2((T)^{\otimes2}X)\to L^2((T^*)^{\otimes2}X)$ and $\sharp:=\flat^{-1}$ given by 
\begin{align*}
\<A^{\sharp}\,|\, T\>_{\m}:=A(T)=: \<A\,|\,T^{\flat}\>_{\m}\quad\m\text{-a.e.},
\end{align*}
 then we can write $|A|_{\rm HS}:=\sqrt{\<A\,|\,A\>_{\m}}$ and $|T|_{\rm HS}:=\sqrt{\<T\,|\,T\>_{\m}}$ for $A\in L^2((T^*)^{\otimes2}X)$ and $T\in L^2((T)^{\otimes2}X)$.

\begin{defn}[{{\bf \boldmath$(2,p)$-Sobolev space $W^{\hspace{0.03cm}2,p}(X)$}}]\label{df:Hessian}
{\rm Let $(X, \mathsf{d}, \m)$ be an $\mathsf{RCD}(K, \infty)$-space with $K\in \R$. Fix $p\in
[\,1,\,+\infty\,[$ and set $q:=p/(p-1)\in\,]1,\,+\infty\,]$}. 
If $p\in [\,2,\, +\infty\,[$, we define the space $W^{\hspace{0.03cm}2,p}(X)$ as all functions $f\in W^{\hspace{0.03cm}1,p}(X)$ 
for which there exists some $A\in L^2((T^*)^{\otimes2}X)$ such that for every $h\in  {\rm Test}^{\infty}(X)_{\rm bs}$ and $g_1$, $g_2\in {\rm Test}^{\infty}(X)$,   
\begin{align}
2\int_X hA&(\nabla g_1,\nabla g_2)\d\m\notag\\
&=-\int_X\<D f,D g_1\>{\rm div}(h\nabla g_2)\d\m-\int_X\<
D f,D g_2\>{\rm div}(h\nabla g_1)\d\m\label{eq:pHessian}\\
&\hspace{1cm} -\int_Xh\<D f,D\<\nabla g_1,\nabla g_2\>\>\d\m,\notag
\end{align}
and $\lvert A\rvert_{\rm HS}\in L^p(X)$. 
If $p\in \,]\,1, \,2\,[$, the space $W^{\hspace{0.03cm}2,p}(X)$ is defined as above but with the additional requirement that $f\in W^{\hspace{0.03cm}1,p}(X)\cap D(\bm{\Delta})$.

If such an $A$ exists, it is unique, denoted by $\Hess\,f$ and called the 
{\it Hessian} of $f$. The space $W^{\hspace{0.03cm}2,p}(X)$ is then endowed with the norm $\lVert \cdot\rVert_{W^{\hspace{0.03cm}2,p}(X)}$ given by 
\begin{align*}
\lVert f\rVert_{W^{\hspace{0.03cm}2,p}(X)}^p:=\lVert f\rVert_{L^p(X)}^p+\lVert \lvert Df\rvert\rVert_{L^p(X)}^p+\lVert\lvert\Hess\,f\rvert_{\rm HS}\rVert_{L^p(X)}^p.
\end{align*}
We also define
\begin{align*}
H^{\hspace{0.03cm}2,p}(X):=\overline{W^{\hspace{0.03cm}2,p}(X)\cap 
{\rm Test}^{\infty}(X)}^{ \lVert \cdot\rVert_{W^{\hspace{0.03cm}2, p}(X)}}.
\end{align*}
\end{defn}

Here 
\lq\lq $\nabla$\rq\rq\, is the gradient operator from
$W^{1,2}(X)$ to $L^2(TX)$ (see~\cite[Definition 2.3.4]{Gigli:NonSmoothDifferentialStr}) and \lq\lq ${\rm div}$\rq\rq\, means the divergence operator from $L^2(TX)$ to $W^{1,2}(X)$ (see~\cite[Definition 2.3.11]{Gigli:NonSmoothDifferentialStr}). The first and second terms in the right hand side of \eqref{eq:pHessian} are well-defined, as by the definition of ${\rm Test}^{\infty}(X)_{\rm bs}$ we have
\begin{align*}
    \<D f,D g_1\>&\in L^p_{\loc}(X),\\
    {\rm div}(h\nabla g_2)
&=h\cdot\Delta g_2+\<\nabla h, \nabla g_2\>
\in 
L^{\infty}(X)_{\rm bs}\subset L^q(X)_{\rm bs},
\end{align*}
then we can apply H\"older's inequality. The third term in  the right hand side of \eqref{eq:pHessian} is well-defined, as by polarization of \cite[Lemma 6.1.9]{GPLecture} we see that $\lvert D\<\nabla g_1,\nabla g_2\>\rvert\in L^2(X)$, hence $\<D f,D\<\nabla g_1,\nabla g_2\>\>\in L^1(X)_{\loc}$ when $p\geq 2$. When $p\in \,]\,1, \,2\,[$, by the Leibniz rule in \cite[{Theorem 4.3.3. vi)}]{GPLecture},
\begin{align*}
    &\int_Xh\<D f,D\<\nabla g_1,\nabla g_2\>\>\d\m\\
    &=\int_X\<D f,D(h\<\nabla g_1,\nabla g_2\>)\>\d\m
    -\int_X\<D h,D f\>\<\nabla g_1,\nabla g_2\>\d\m\\
    &=-\int_X h\<\nabla g_1,\nabla g_2\>\d\bm{\Delta}f
    -\int_X\<D h,D f\>\<\nabla g_1,\nabla g_2\>\d\m
\end{align*}
which is also finite.  
\begin{remark}
    When $p=2$, the definition given above is \emph{a priori} weaker than the definition of $W^{2,2}(X)$ given in \cite[Definition~3.3.1]{Gigli:NonSmoothDifferentialStr} or \cite[Definition~6.2.6]{GPLecture}: the difference being that the class of test functions is taken to be ${\rm Test}(X)$ and ${\rm Test}^\infty(X)$ respectively, in contrast to our definition which takes $h\in {\rm Test}^\infty(X)_{\rm bs}$. Since ${\rm Test}^\infty(X)$ is dense in $W^{1, 2}(X)$ by \cite[Proposition 6.1.8]{GPLecture} and $\Div(h\nabla g)=h\cdot \Delta g+\<\nabla h, \nabla g\>$ for $h$, $g\in {\rm Test}(X)$, we see \cite[Definition~3.3.1]{Gigli:NonSmoothDifferentialStr} and \cite[Definition~6.2.6]{GPLecture} are equivalent, while equivalence with our Definition~\ref{df:Hessian} then follows from Lemma~\ref{lem:denseness}.

    The additional restriction that $f\in D(\bm{\Delta})$ when $p\in ]1, 2[$ is  naturally satisfied in the case when $(X, \mathsf{d}, \m)$ is a Riemannian manifold equipped with the geodesic distance and canonical volume. We have not treated the case $p=1$: as mentioned above, the notion of minimal $p$-weak upper gradient is independent of $p$ when $p\in \,]\,1, \,\infty\,[$, but this is not known when $p=1$, hence it is not clear if the object $\< D f, D g \>$ is well-defined.
\end{remark}

\subsection{{Modified \boldmath$(2,p)$}-Sobolev spaces}\label{subsec: alt(2,p)Sobolev}
As mentioned in the introduction, the only elements of ${\rm Test}(X)\cap W^{2, p}(X)$ given by Definition~\ref{df:Hessian} may be constant functions, hence the usual method to show $W^{2, p}(X)$ contains many functions fails, in the context of $\mathsf{RCD}(K,N)$-spaces. 
Thus, we now introduce the modified $(2,p)$-Sobolev space $W^{\hspace{0.03cm}2, p}_\ast(X)$ which can be defined for  $p\in\,[1,\,+\infty\,[$ in the framework of general metric measure spaces.

\begin{defn}\label{df:modifiedSobolev}
If $(X, \mathsf{d}, \m)$ is a metric measure space, for $p\in \,[1,\,+\infty\,[$, we define the Sobolev space $W^{\hspace{0.03cm}2, p}_\ast(X)$ by
\begin{align*}
    W^{\hspace{0.03cm}2, p}_\ast(X)
    :=\{f\in W^{\hspace{0.03cm}1, p}(X)\mid\, \exists G\in \mathcal{UG}^{\,p}(f)\cap W^{1, p}(X)\},
\end{align*}
and also
\begin{align}
    \lVert f\rVert_{W^{\hspace{0.03cm}2, p}_\ast(X)}
    :=
   \left( \lVert f\rVert^p_{L^p(X)}+\inf_{G\in \mathcal{UG}^{\,p}(f)\cap W^{1, p}(X)}\lVert  G\rVert^p_{W^{1,p}(X)}\right)^{\frac{1}{p}}.\label{eq:distanceSobolev}
\end{align}
If $(X, \mathsf{d}, \m)$ is also an $\mathsf{RCD}(K, \infty)$-space with $K\in \R$, we also define
\begin{align*}
H^{\hspace{0.03cm}2,p}_\ast(X):=\overline{W^{\hspace{0.03cm}2,p}_\ast(X)\cap 
{\rm Test}^{\infty}(X)}^{ \lVert \cdot\rVert_{W^{\hspace{0.03cm}2, p}_\ast(X)}}.
\end{align*}
\end{defn}
Note that $\mathcal{UG}^{\,p}(f)\cap W^{1, p}(X)$ is a closed, convex subset of $W^{1,p}(X)$ in view of Lemma~\ref{lem:stability} and \cite[Proposition~2.7]{Gigli:OntheDifferentialStr}. 
When $(X,\mathsf{d},\m)$ is an $\mathsf{RCD}(K,\infty)$-space and 
$p\in\,]\,1,\,+\infty\,[$, by Corollary~\ref{cor:minimizer}, 
we see that there exists a unique $G_f^*\in \mathcal{UG}^{\,p}(f)\cap W^{1, p}(X)$ such that 
\begin{align*}
\lVert  G_f^*\rVert^p_{W^{1,p}(X)}=\inf_{G\in \mathcal{UG}^{\,p}(f)\cap W^{1, p}(X)}\lVert  G\rVert^p_{W^{1,p}(X)}.
\end{align*}

\begin{prop}\label{prop: Banach space}
 $(W^{\hspace{0.03cm}2, p}_\ast(X), \lVert \cdot\rVert_{W^{\hspace{0.03cm}2, p}_\ast(X)})$ is a Banach space for any metric measure space $(X, \mathsf{d}, \m)$.
\end{prop}
\begin{pf}
For any $\lambda\in \R\setminus \{0\}$ and $f\in W^{\hspace{0.03cm}2, p}_\ast(X)$, by definition we see that $\mathcal{UG}^{\,p}(\lambda f)=\{\lvert\lambda\rvert G\mid G\in \mathcal{UG}^{\,p}(f)\}$, since $W^{1, p}(X)$ is closed under multiplication by real numbers, and of course, $0=0f\in W^{\hspace{0.03cm}2, p}_\ast(X)$, this implies  $W^{\hspace{0.03cm}2, p}_\ast(X)$ is as well. Additionally, since $G_i\in \mathcal{UG}^{\,p}(f_i)\cap W^{1, p}(X)$ for $i=1$, $2$ implies $G_1+G_2\in \mathcal{UG}^{\,p}(f_1+f_2)\cap W^{1, p}(X)$ for any measurable $f_1$ and $f_2$, we see $W^{\hspace{0.03cm}2, p}_\ast(X)$ is also closed under addition, meaning it is a vector space.

Next we show $\lVert \cdot\rVert_{W^{\hspace{0.03cm}2, p}_\ast(X)}$ is a norm. It is clear that $\lVert f\rVert_{W^{\hspace{0.03cm}2, p}_\ast(X)}\geq 0$ for any $f\in W^{\hspace{0.03cm}2, p}_\ast(X)$ and $\lVert 0\rVert_{W^{\hspace{0.03cm}2, p}_\ast(X)}=0$. Since $\lVert f\rVert_{L^p(X)}\leq \lVert f\rVert_{W^{\hspace{0.03cm}2, p}_\ast(X)}$, we see that $\lVert f\rVert_{W^{\hspace{0.03cm}2, p}_\ast(X)}=0$ if and only if $\m$-a.e. we have $f=0$. If $\lambda\in \R$ and $f\in W^{\hspace{0.03cm}2, p}_\ast(X)$, we have
\begin{align*}
 \lVert \lambda f\rVert_{W^{\hspace{0.03cm}2, p}_\ast(X)}^p
 =\lVert \lambda f\rVert_{L^p(X)}^p+\inf_{G\in \mathcal{UG}^{\,p}(f)\cap W^{1, p}(X)}\lVert \lambda G\rVert_{W^{1, p}(X)}^p,
\end{align*}
which yields $ \lVert \lambda f\rVert_{W^{\hspace{0.03cm}2, p}_\ast(X)}= \lvert \lambda \rvert\lVert  f\rVert_{W^{\hspace{0.03cm}2, p}_\ast(X)}$. Finally, suppose $f_1$, $f_2\in W^{\hspace{0.03cm}2, p}_\ast(X)$ and fix $\varepsilon>0$. Then for $i=1$, $2$, there exist $G^\varepsilon_i\in \mathcal{UG}^{\,p}(f_i)\cap W^{1, p}(X)$ such that
\begin{align*}
\left( \lVert f_i\rVert^p_{L^p(X)}+\lVert  G^\varepsilon_i\rVert^p_{W^{1,p}(X)}\right)^{\frac{1}{p}}
\leq  \lVert f_i\rVert_{W^{\hspace{0.03cm}2, p}_\ast(X)}+\varepsilon.
\end{align*}
Then since $G^\varepsilon_1+G^\varepsilon_2\in \mathcal{UG}^{\,p}(f_1+f_2)\cap W^{1, p}(X)$, by the triangle inequality for $L^p$ and $\ell^p$ norms we have
\begin{align*}
  \lVert f_1+f_2\rVert_{W^{\hspace{0.03cm}2, p}_\ast(X)}
  &\leq \left( \lVert f_1+f_2\rVert^p_{L^p(X)}+\lVert  G^\varepsilon_1+G^\varepsilon_2\rVert^p_{W^{1,p}(X)}\right)^{\frac{1}{p}}\\
  &\leq \left( \lVert f_1\rVert^p_{L^p(X)}+\lVert  G^\varepsilon_1\rVert^p_{W^{1,p}(X)}\right)^{\frac{1}{p}}
  +\left( \lVert f_2\rVert^p_{L^p(X)}+\lVert  G^\varepsilon_2\rVert^p_{W^{1,p}(X)}\right)^{\frac{1}{p}}\\
  &\leq  \lVert f_1\rVert_{W^{\hspace{0.03cm}2, p}_\ast(X)}+\lVert f_2\rVert_{W^{\hspace{0.03cm}2, p}_\ast(X)}+2\varepsilon,
\end{align*}
and we obtain the triangle inequality by taking $\varepsilon\to 0$. Thus $\lVert \cdot\rVert_{W^{\hspace{0.03cm}2, p}_\ast(X)}$ is a norm.
    
We now show completeness, to this end suppose that $\{f_k\}_{k\in \mathbb{N}}$ is a $\lVert \cdot\rVert_{W^{\hspace{0.03cm}2, p}_\ast(X)}$-Cauchy sequence in $W^{\hspace{0.03cm}2, p}_\ast(X)$. Since the sequence is clearly Cauchy in $\lVert \cdot\rVert_{L^p(X)}$, there exists some $f\in L^p(X)$ with $\lVert f_k-f\rVert_{L^p(X)}\to 0$ as $k\to\infty$. For any $\varepsilon>0$, there exists $K_\varepsilon\in \N$ and $G_{k, \tilde{k}}\in \mathcal{UG}^{\,p}(f_k-f_{\tilde{k}})\cap W^{1, p}(X)$ for any $k$ and $\tilde{k}\geq K_\varepsilon$ such that
\begin{align*}
 \lVert f_k-f\rVert_{L^p(X)}<\frac{\varepsilon}{2},
 \qquad \lVert G_{k, \tilde{k}}\rVert_{W^{1, p}(X)}<\frac{\varepsilon}{4}.
\end{align*}
Fix some $\varepsilon>0$ and $k\geq K_\varepsilon$, then we can take a subsequence $\{k_\ell\}$ such that $k_1:=k$ and $\lVert G_{k_\ell, k_{\ell+1}}\rVert_{W^{1, p}(X)}<\frac{\varepsilon}{2^{1+\ell}}$, and also such that $\{f_{k_\ell}\}_{\ell\in \N}$ converges pointwise $\m$-a.e. to $f$. Defining $G^k_L:=\sum_{\ell=1}^L G_{k_\ell, k_{\ell+1}}$ for each $L\in \N$, we can see that $\{G^k_L\}_{L\in \N}$ is a Cauchy sequence in $W^{1, p}(X)$, hence converges to some $G$ in $W^{1, p}(X)$. Additionally, we see $G^k_L\in \mathcal{UG}^{\,p}\left(\sum_{\ell=1}^L(f_{k_\ell}-f_{k_{\ell+1}})\right)\cap L^p(X)=\mathcal{UG}^{\,p}(f_{k}-f_{k_{L+1}})\cap L^p(X)$ for each $L\in \N$, thus by Lemma~\ref{lem:stability} when $p=1$ and \cite[Proposition 2.7]{Gigli:OntheDifferentialStr} when $p>1$, we see that $G\in \mathcal{UG}^{\,p}(f_k-f)\cap L^p(X)$. Thus we obtain
\begin{align*}
 \lVert f_k-f\rVert_{W^{\hspace{0.03cm}2, p}_\ast(X)}^p
 &\leq  \lVert f_k-f\rVert_{L^p(X)}^p+\lVert G^k\rVert_{W^{1, p}(X)}
 <\frac{\varepsilon}{2}+\varliminf_{L\to\infty}\sum_{\ell=1}^L\lVert G_{k_\ell, k_{\ell+1}}\rVert_{W^{1, p}(X)}\\
 &<\frac{\varepsilon}{2}+\sum_{\ell=1}^\infty\frac{\varepsilon}{2^{1+\ell}}=\varepsilon,
\end{align*}
holding for any $k\geq K_\varepsilon$, finishing the proof of completeness.
\end{pf}

The above completeness, in particular, shows that $H^{\hspace{0.03cm}2,p}_\ast(X )\subset W^{\hspace{0.03cm}2,p}_\ast(X)$ when $(X,\mathsf{d},\m)$ is an 
$\mathsf{RCD}(K,\infty)$-space. We now show a relationship between the spaces $W^{\hspace{0.03cm}2, p}(X)$ and $W^{\hspace{0.03cm}2, p}_\ast(X)$.

\begin{prop}\label{prop: continuous embedding}
Suppose that $(X, \mathsf{d}, \m)$ is an $\mathsf{RCD}(K, N)$-space with $K\in \R$ and $N\in [\,1,\, +\infty\,[$. Then, for $p\in ]\,1,\,+\infty\,[$
    \begin{align*}
        H^{\hspace{0.03cm}2, p}(X)\subset W^{\hspace{0.03cm}2, p}_\ast(X),
    \end{align*}
    in particular, 
    \begin{align*}
        W^{\hspace{0.03cm}2, p}(X)\cap {\rm Test}^{\infty}(X)\subset W^{\hspace{0.03cm}2, p}_\ast(X). 
    \end{align*}
    Additionally, the above inclusions are $1$-Lipschitz embeddings.
\end{prop}
\begin{pf}
    Take $f\in H^{\hspace{0.03cm}2,p}(X)$, then there exists a sequence $\{f_n\}\subset W^{\hspace{0.03cm}2,p}(X)\cap {\rm Test}^{\infty}(X)$ such that $f_n\to f$ in $W^{\hspace{0.03cm}2,p}(X)$. By the Kato-type inequality \cite[Lemma 4.6]{GigliViolo23}, 
    \begin{align*}
    \frac{t+N}{t+N-1}\lvert D\lvert Df_n\rvert\rvert^2\leq\lvert \Hess f_n\rvert_{\rm HS}^2+\frac{({\rm Tr}{\Hess} f_n)^2}{t}\quad \m\text{-a.e.~for each }t>0,
    \end{align*}
    thus letting $t\to\infty$, we obtain
    \begin{align*}
    \lvert D\lvert Df_n\rvert\rvert\leq\lvert \Hess f_n\rvert_{\rm HS}\quad \m\text{-a.e.},
    \end{align*}
    hence, 
    \begin{align*}
        \lVert \lvert D\lvert Df_n\rvert\rvert\rVert_{L^p(X)}^p
        &\leq \lVert \lvert \Hess f_n\rvert_{\rm HS}\rVert_{L^p(X)}^p
    \end{align*}
    Since $\{\lvert Df_n\rvert\}$ (resp.~$\{\lvert {\rm Hess} f_n\rvert_{\rm HS}\}$) converges to $\lvert Df\rvert$ 
    (resp.~$\lvert {\rm Hess} f\rvert_{\rm HS}$) in $L^p$ and the $p$-energy $\lvert Df\rvert\mapsto \lVert \lvert D\lvert Df\rvert\rvert\lVert_{L^p}^p$ is lower semicontinuous with respect to $L^p$, we obtain
    \begin{align*}
\lVert\lvert D\lvert Df\rvert\rvert\rVert_{L^p(X)}
&\leq \lVert \lvert \Hess f\rvert_{\rm HS}\rVert_{L^p(X)}<+\infty,
    \end{align*}
    which implies $f\in W_*^{\hspace{0.03cm}2,p}(X)$. Since $\lVert\lvert D\lvert Df\rvert\rvert\in \mathcal{UG}^p(f)$, by the definition \eqref{eq:distanceSobolev} the above yields $\lVert f \rVert_{W_*^{\hspace{0.03cm}2,p}(X)}\leq \lVert f\rVert_{W^{\hspace{0.03cm}2, p}(X)}$.
    \end{pf}

\section{Proof of Theorem~\ref{thm: W2,p morrey}}\label{sec: proof of main theorem}
 In this section, we give the proof of Theorem~\ref{thm: W2,p morrey}. The proof is based on the following version of local, first order Sobolev and Morrey inequalities, which mainly follows from \cite{HK}. Note that in the theorem below, there are two powers $\uppergradpower$ and $\sobolevpower$. The Sobolev and Morrey inequalities will be stated for pairs $(f, G)$ with $G$ being a $\uppergradpower$-weak upper gradient of $f$, but the inequality itself will involve $L^{\sobolevpower}$ norms; in particular the expressions on the right hand sides of~\eqref{item: sobolev ineq} and~\eqref{item: morrey ineq} below may be infinite, and we do not necessarily require either of $f$ or $G$ to belong to $L^p(X)$ for any value of $p$.
 
\begin{thm}\label{thm: W1p sobolev}
    Suppose $(X,\mathsf{d},\m)$ is a $\mathsf{CD}(K,N)$-space with $K\in \R$ and $N\in [\,1,\,+\infty\,[$. 
    Also fix $x_0\in X$ and $r_0>0$, let $\Omega:=B_{r_0}(x_0)$, and fix $\uppergradpower$, $\sobolevpower\in [\,1,\, +\infty\,[$, a Borel function $f$, and $G\in \mathcal{UG}^{\,\uppergradpower}(f)$. Then the following hold:
    \begin{enumerate}[\rm (1)]
        \item\label{item: sobolev ineq} If $\sobolevpower\in [\,1,\, N\,[$, there exists some $C_S>0$ such that 
        \begin{align*}
            \left(\m(\Omega)^{-1}\int_{\Omega}\lvert f-f_{\Omega}\rvert^{\sobolevpower^\ast}\d\m\right)^{\frac{1}{\sobolevpower^\ast}}
            \leq C_S r_0\left(\m(\Omega)^{-1}\int_{\Omega}\lvert G\rvert^{\sobolevpower}\d\m\right)^{\frac{1}{\sobolevpower}}
        \end{align*}
        where $\sobolevpower^\ast:=\frac{N\sobolevpower}{N-\sobolevpower}$ and $f_{\Omega}:=\m(\Omega)^{-1}\int_{\Omega} f\d\m$. 
        \item\label{item: morrey equal} If
         $\sobolevpower=N$, there exist $C_1>0$ and $C_2>0$ such that 
         \begin{align*}
            \m(\Omega)^{-1}\int_{\Omega}
            \exp\left( \dfrac{C_1\m(\Omega)^{1/N}\lvert f-f_{\Omega}\rvert}{r_0
            \|G\|_{L^N(\Omega)}
            } \right)
            \d\m
            \leq C_2 
        \end{align*}
   \item\label{item: morrey ineq} If $\sobolevpower\in \,]\,N, \,+\infty\,[$ and $G\in L^{\sobolevpower}(\Omega)$, then $f$ has a continuous $\m$-version on $\Omega$ satisfying
        \begin{align*}
            \lvert f(y_1)-f(y_2)\rvert
            \leq C_Mr_0^{\frac{N}{\sobolevpower}}\mathsf{d}(y_1, y_2)^{1-\frac{N}{\sobolevpower}}\left(\m(\Omega)^{-1}\int_{\Omega}\lvert G\rvert^{\sobolevpower}\d\m\right)^{\frac{1}{\sobolevpower}}
        \end{align*}
        for any $y_1$, $y_2\in \Omega$ and some constant $C_M>0$.

    \end{enumerate}
    Here $C_S$, $C_1$, $C_2$, and $C_M$ depend on $K$, $N$, $x_0$, $r_0$, $\uppergradpower$, and $\sobolevpower$.
\end{thm}
 To obtain the estimate in Theorem~\ref{thm: W1p sobolev}~\eqref{item: sobolev ineq} when $\sobolevpower=1$, we need a few preliminary results. The first is to show the so-called truncation property as in \cite[Definition, p.~9]{HK} for a pair $f\in W^{1, 1}(X)$, $G\in \mathcal{UG}^{\,1}(f)\cap L^1(X)$.

We note the following Lemma~\ref{lem:compositeUG} and Corollary~\ref{lem:compositeUG*} do not require the space to satisfy the $\mathsf{CD}(K, N)$ condition.
\begin{lem}\label{lem:compositeUG}
    Let $f:X\to \R$ be a Borel function and fix $p\in\,[\,1,\,+\infty\,[$.
 Also, assume that $\phi$ is a unit contraction, i.e.,
 $\phi(t):=0\lor t\land 1$ and $f\in S^{\hspace{0.03cm}p}(X)$ with $G\in\mathcal{UG}^{\,p}(f)\cap L^p(X)$. 
    Then $\phi\circ f\in S^{\hspace{0.03cm}p}(X)$ and $G\mathds{1}_{\{0<f\leq1\}}\in\mathcal{UG}^{\,p}(\phi\circ f)\cap L^p(X)$. 
\end{lem}
\begin{proof}[{\bf Proof.}]
Let $\phi_{\varepsilon}$ be smooth functions constructed by mollifying $\phi$ (see \cite[Exercise~1.2.1]{FOT}) such that $\lvert \phi'_{\varepsilon}\rvert\leq 1$, $-\varepsilon\leq \phi_\varepsilon\leq 1+\varepsilon$, and $\phi_{\varepsilon}(0)=0$ for all $\varepsilon>0$, while $\phi_{\varepsilon}(t)\to \phi(t)$ and $\phi_{\varepsilon}'(t)\to \mathds{1}_{]\,0,\,1\,]}(t)$ for a.e.~$t$ as $\varepsilon\to 0$. Since $\phi_{\varepsilon}$ is a smooth function satisfying $\phi_{\varepsilon}(0)=0$, we can see that
$\phi_{\varepsilon}\circ f\in S^p(X)$ and $\lvert \phi_{\varepsilon}'\circ f\rvert G\in\mathcal{UG}^{\,p}(\phi_{\varepsilon}\circ f)\cap L^p(X)$. 
Indeed, when $p>1$ this follows directly from the chain rule, \cite[Theorem 2.1.28 B2)]{GPLecture}. When $p=1$, we note that one only needs to verify that Steps 1, 2, 3, and 7 from the proof of \cite[Theorem 2.1.28 B2)]{GPLecture} hold: Steps 1, 2, and 7 clearly still hold when $p=1$, while Step 3 relies on the characterization \cite[Theorem 2.1.21]{GPLecture}, which can be seen remains true when $p=1$. 
In particular, we have
\begin{align}\label{eqn: chain rule}
\int_{C([\,0,\,1\,], X)}\lvert\phi_{\varepsilon}(f(\gamma_1))
-\phi_{\varepsilon}(f(\gamma_0))\rvert
\text{\boldmath$\pi$}(\d\gamma)\leq
\int_{C([\,0,\,1\,], X)}\int_0^1\lvert\phi_{\varepsilon}'(f(\gamma_t))\rvert G(\gamma_t)\lvert\dot{\gamma}_t\rvert\d t\text{\boldmath$\pi$}(\d\gamma).
\end{align}
Since $\lvert\phi_{\varepsilon}(t)-\phi_{\varepsilon}(s)\rvert\leq \lvert t-s\rvert$, by dominated convergence theorem the sequence of integrals on the left of~\eqref{eqn: chain rule} converges to $\int_{C([\,0,\,1\,], X)}\lvert\phi(f(\gamma_1))-\phi(f(\gamma_0))\rvert
\text{\boldmath$\pi$}(\d\gamma)$. On the other hand, if $p=1$, since $\bm{\pi}$ is an $\infty$-test plan there is a constant $C_{\bm{\pi}}>0$ such that $\lvert \dot \gamma_t\rvert\leq C_{\bm{\pi}}$ for $\bm{\pi}$-a.e. $\gamma$. Thus for any $\varepsilon>0$ the integrand on the right of~\eqref{eqn: chain rule} is bounded for all $t\in [0, 1]$ and $\bm{\pi}$-a.e. $\gamma$ by $C_{\bm{\pi}}G(\gamma_t)$, and using the bounded compression of $\bm{\pi}$ this satisfies
\begin{align*}
    \int_0^1\int_X G\d((\mathsf{e}_t)_\sharp{\bm{\pi}})\d t
    &\leq  C\int_0^1\lVert G\rVert_{L^1(X)}\d t
    =C\lVert  G\rVert_{L^1(X)}.
\end{align*}
If $p>1$, the integrand is bounded by $G(\gamma_t)\lvert \dot \gamma_t\rvert$ and we have (again using bounded compression)
\begin{align*}
    \int_{C([\,0,\,1\,], X)}&\int_0^1 G(\gamma_t)\lvert\dot{\gamma}_t\rvert\d t\text{\boldmath$\pi$}(\d\gamma)\\
    &\leq \left(\int_{C([\,0,\,1\,], X)}\int_0^1 (G^p\circ \gamma_t)\d t\bm{\pi}(\d\gamma)\right)^{\frac{1}{p}} \left(\int_{C([\,0,\,1\,], X)}\mathsf{Ke}_q(\gamma)\bm{\pi}(\d\gamma)\right)^{\frac{1}{q}}\\
    &= \left(\int_0^1 \int_XG^p\d((\mathsf{e}_t)_\sharp{\bm{\pi}})\d t\right)^{\frac{1}{p}} 
    \left(\int_{C([\,0,\,1\,], X)}\mathsf{Ke}_q(\gamma)\bm{\pi}(\d\gamma)\right)^{\frac{1}{q}}\\
    &\leq C^{\frac{1}{p}}\lVert G\rVert_{L^p(X)}\left(\int_{C([\,0,\,1\,], X)}\mathsf{Ke}_q(\gamma)\bm{\pi}(\d\gamma)\right)^{\frac{1}{q}};
\end{align*}
in all cases we can apply the dominated convergence theorem to obtain the conclusion. 
\end{proof}
\begin{cor}\label{lem:compositeUG*}
Let $f\in S^{\hspace{0.03cm}p}(X)$ with $G\in\mathcal{UG}^{\,p}(f)\cap L^p(X)$. 
    Set $p\in\,[\,1,\,+\infty\,[$ and $q:=p/(p-1)$ if $p>1$ and $q:=\infty$ if $p=1$. 
    Take $t_1,t_2\in\R$ with $t_1<t_2$ and consider 
    $\phi^{t_2}_{t_1}(t)=0\lor(t-t_1)\land(t_2-t_1)$. 
    Then $\phi^{t_2}_{t_1}\circ f\in S^p(X)$ and $G\mathds{1}_{\{t_1<f\leq t_2\}}\in\mathcal{UG}^{\,p}(\phi^{t_2}_{t_1}\circ f)\cap L^p(X)$. 
\end{cor}
\begin{proof}[{\bf Proof.}]
    Since $f\in S^{\hspace{0.03cm}p}(X)$ with $G\in\mathcal{UG}^{\,p}(f)\cap L^p(X)$, we have 
    $f-t_1\in S^{\hspace{0.03cm}p}(X)$ with $G\in\mathcal{UG}^{\,p}(f-t_1)\cap L^p(X)$, hence $\frac{f-t_1}{t_2-t_1}\in S^{\hspace{0.03cm}p}(X)$ with $\frac{G}{t_2-t_1}\in\mathcal{UG}^{\,p}\left(\frac{f-t_1}{t_2-t_1}\right)\cap L^p(X)$. Then by Lemma~\ref{lem:compositeUG}, 
    $\phi_0^1\left(\frac{f-t_1}{t_2-t_1}\right)\in S^{\hspace{0.03cm}p}(X)$ with 
    $\frac{G}{t_2-t_1}\mathds{1}_{]\,0,\,1\,]}\left(\frac{f-t_1}{t_2-t_1}\right)\in 
    \mathcal{UG}^{\,p}\left(
    \phi_0^1\left(\frac{f-t_1}{t_2-t_1}\right)\right)\cap L^p(X)$.
    From $\phi^{t_2}_{t_1}(t)=(t_2-t_1)\phi^{1}_{0}(\frac{t-t_1}{t_2-t_1})$, we can deduce the assertion. 
\end{proof}

\begin{defn}
    If $\Omega\subset X$, a positive Borel measure $\mu$ on $X$ is said to be \emph{doubling on $\Omega$} if there exists a constant $C>0$ such that for any $x\in \Omega$ and $0<r\leq 5\diam(\Omega)$, it holds that
    \begin{align*}
        \mu(B_{2r}(x))\leq C\mu(B_r(x)).
    \end{align*}
\end{defn}
\begin{proof}[\bf Proof of Theorem~\ref{thm: W1p sobolev}]

We first note that regardless of the value of $\uppergradpower\in [\,1,\, +\infty\,[$, the pair $(f, G)$ satisfies a local $(1, \sobolevpower\,)$-Poincar{\'e} inequality; meaning there exists a $C_P>0$ such that for any $x\in \Omega$ and $r>0$ such that $B_{2r}(x)\subset \Omega$,
    \begin{align}\label{eqn: poincare}
            \m(B_r(x))^{-1}\int_{B_r(x)}\lvert f-f_{B_r(x)}\rvert \d\m
            \leq C_P r\left(\m(B_{2r}(x))^{-1}\int_{B_{2r}(x)}\lvert G\rvert^{\sobolevpower} \d\m\right)^{\frac{1}{\sobolevpower}}.
        \end{align}
        Indeed, fix an open ball $B$, and let {\boldmath$\pi$} be the measure constructed as in \cite[Proof of Theorem 1]{Rajala12} from $B$ which exists as $(X, \mathsf{d}, \m)$ is assumed to be a $\mathsf{CD}(K, N)$-space. By \cite[Lemma 1]{Rajala12} we see that {\boldmath$\pi$} satisfies Definition~\ref{def:$q$-test}~\eqref{item:qtest1}. At the same time, {\boldmath$\pi$}-a.e. curve $\gamma$ is a geodesic with endpoints in $B$, hence for such a curve by \cite[Proposition 1.2.10]{GPLecture},
    \begin{align*}
        \lip(\gamma)&=\mathsf{d}(\gamma_0, \gamma_1)\leq \diam(B),\\
        \mathsf{Ke}_q(\gamma)&=\mathsf{d}(\gamma_0, \gamma_1)^q\leq \diam(B)^q,\qquad \forall q\in \,]\,1,\, +\infty\,[
    \end{align*}
    thus, we find that {\boldmath$\pi$} is a $q$-test plan for any $q\in \,]\,1,\, +\infty\,]$. Hence in the final calculation in \cite[Proof of Theorem 1]{Rajala12}, one can replace $g$ there by any $G\in \mathcal{UG}^{\,\uppergradpower}(f)$, then the fact that $N<\infty$ implies that $(f, G)$ satisfies a local $(1, 1)$-Poincar{\'e} inequality (that is,~\eqref{eqn: poincare} with $\sobolevpower=1$ on the right hand side) in the same way one obtains \cite[Theorem 2]{Rajala12} from \cite[Theorem 1]{Rajala12}. Then~\eqref{eqn: poincare} with general $\sobolevpower$ follows by H\"older's inequality.

    Next, by the Bishop--Gromov inequality 
    (see Remark~\ref{rem:CD}~\eqref{item: bishop gromov}) we can see that $\m$ is doubling on $\Omega$ with a constant depending on $K$, $N$, and the radius of $\Omega$ (see \cite[Corollary 30.14]{Villani03book}), and also for any $x\in \Omega$ and $r\leq \diam(\Omega)$ we have
    \begin{align*}
     \m(B_r(x))\geq C\left(\frac{r}{\diam(\Omega)}\right)^N\m(\Omega)
    \end{align*}
    for some $C>0$, depending on $K$, $N$, and $\Omega$.
    Since $(X, \mathsf{d})$ is a geodesic space, from  \cite[Corollary 9.5]{HK} we can apply \cite[Theorem 9.7]{HK} to obtain the conclusion of Theorem~\ref{thm: W1p sobolev} for $\sobolevpower>1$ without using the truncation property of the pair $f$ and $G$.  
 Finally, we can still apply \cite[Theorem 9.7]{HK} in the case $\sobolevpower=1$ if $G\in L^1(\Omega)$, as the pair $f$ and $G$ satisfy the truncation property \cite[Definition, p.~9]{HK}, by Corollary~\ref{lem:compositeUG*}, and the inequality is trivially true if $G\not\in L^1(\Omega)$.
\end{proof}

We are now ready to give the proof of the higher order Morrey's inequality.

\begin{apf}{Theorem~\ref{thm: W2,p morrey}}
     Let $p\in[\,1,\,+\infty\,[$, and fix $f\in W^{\hspace{0.03cm}2, p}_\ast(X)$ and an open ball $B\subset X$; by definition, there exists some $G\in \mathcal{UG}^{\,p}(f)\cap W^{1, p}(X)$. 

     \begin{enumerate}[{\rm (1)}]
     \item  Suppose $p>N$. 
     Since $W^{2,p}_*(X)\subset W^{1,p}(X)$, 
     the assertion follows from Theorem~\ref{thm: W1p sobolev}~\eqref{item: morrey ineq} directly.   
  \item Suppose $p=N$. Since $G\in W^{1, p}(X)$, we may apply Theorem~\ref{thm: W1p sobolev}~\eqref{item: morrey equal} to the pair $(G, \lvert DG\rvert)$ with $\uppergradpower=\sobolevpower=p=N$ there, combined with triangle inequality for $L^N$ and the inequality $x^{\,q}\leq 
 C_qe^x$ $(x\geq0)$ with 
 \begin{align*}
 C_q:=\max\left\{\dfrac{\Gamma([q]+2)}{\Gamma(q+1)}, \dfrac{\Gamma([q]+1)}{\Gamma(q+1)} \right\}    
 \end{align*}
 via $\Gamma$-functions to obtain, 
  for any $q\geq0$
     \begin{align*}
         \lVert G\rVert_{L^{\,q}(B)}
         \leq C\left(\lVert G\rVert_{L^N(X)}+\lVert \lvert DG\rvert\rVert_{L^N(X)}\right)
     \end{align*}
     where $C>0$ depends on the ball $B$ and $q$, in particular $G\in L^q(B)$ for any $q\geq0$. For any $q>N$, we can apply Theorem~\ref{thm: W1p sobolev}~\eqref{item: morrey ineq} with $\uppergradpower=p$ and $\sobolevpower=q$  there to the pair $(f, G)$ which yields the theorem.
     
    \item  Suppose $N/2<p<N$. 
     Since $p<N$ and $G\in W^{1, p}(X)$, we may apply Theorem~\ref{thm: W1p sobolev}~\eqref{item: sobolev ineq} to the pair $(G, \lvert DG\rvert)$ with $\uppergradpower=\sobolevpower=p$ there, combined with the triangle inequality for $L^p$ norms to obtain
     \begin{align*}
         \lVert G\rVert_{L^{\frac{Np}{N-p}}(B)}
         \leq C\left(\lVert G\rVert_{L^p(X)}+\lVert \lvert DG\rvert\rVert_{L^p(X)}\right)
     \end{align*}
     where $C>0$ depends on the ball $B$, in particular $G\in L^{\frac{Np}{N-p}}(B)$. As $p>N/2$, we can see $\frac{Np}{N-p}>N$ and
    \begin{align*}
        1-\frac{N}{\frac{Np}{N-p}}=2-\frac{N}{p}>0,
    \end{align*}
    thus applying Theorem~\ref{thm: W1p sobolev}~\eqref{item: morrey ineq} with $\uppergradpower=p$ and $\sobolevpower=\frac{Np}{N-p}$  there to the pair $(f, G)$ yields the theorem.
     \end{enumerate}
\end{apf}

\section{Existence of  \texorpdfstring{$W^{\hspace{0.03cm}2, p}_\ast(X)$}{} and \texorpdfstring{$H^{\hspace{0.03cm}2, p}_\ast(X)$}{} functions}\label{sec: W2,p functions} 
In this section, we will discuss existence of functions in $H^{\hspace{0.03cm}2, p}_\ast(X)$ and $W^{\hspace{0.03cm}2, p}_\ast(X)$, under different conditions on 
metric measure space 
$(X, \mathsf{d}, \m)$.
\subsection{There are many $W^{\hspace{0.03cm}2, p}_\ast(X)$ functions}

In this subsection, we do not require any curvature-dimension conditions. Fix $p\in [\,1, \,+\infty\,[$. First assume that $\m$ has finite $p$th moment, meaning that for some point $x_0\in X$ 
\begin{align*}
    \int_X \mathsf{d}(x, x_0)^p \d\m(x)<\infty.
\end{align*}
Note that since open balls have finite $\m$-measure, this implies 
\begin{align*}
    \m(X)
    &= \m(X\setminus B_1(x_0))+\m(B_1(x_0))
    \leq \int_{X\setminus B_1(x_0)} \mathsf{d}(x, x_0)^p \d\m(x)+\m(B_1(x_0))<\infty,
\end{align*}
thus, we also have $\int_X \mathsf{d}(x, \tilde x_0)^p \d\m(x)<\infty$ for any other point $\tilde x_0\in X$. Then we see that $\lip(X)\subset W_*^{\hspace{0.03cm}2,p}(X)$. Indeed, if $f\in \lip(X)$ then it is clear that $G\equiv \lip(f)$ is a $p$-weak upper gradient of $f$ for any $p\in [\,1,\, +\infty\,[$ with $G\in W^{1, p}(X)$, while we have
\begin{align*}
    \int_X \lvert f\rvert^p\d\m
    &\leq 2^{p-1}\int_X (\lip(f)\mathsf{d}(x, x_0)^p+\lvert f(x_0)\rvert^p)\d\m(x)<\infty,
\end{align*}
hence $f\in W^{1, p}(X)$.

For general $\m$, we can show that $\lip(X)_{\rm bs}\subset W_*^{\hspace{0.03cm}2,p}(X)$. To see this, take $f\in \lip(X)_{\rm bs}$ which is zero outside a ball $B_{r_0}(x_0)$, then by \cite[Lemma 6.2.14]{GPLecture} there exists a nonnegative function $G\in W^{1, p}(X)$ such that $G\equiv \lip(f)$ on $B_{r_0}(x_0)$ and $G\equiv 0$ on $X\setminus B_{2r_0}(x_0)$. Now, let $\gamma\in C([\,0,\,1\,], X)$ be arbitrary. If $\gamma_0$, $\gamma_1\in B_{r_0}(x_0)$, by continuity we have 
\begin{align*}
    t_0:&=\sup\{t\in [\,0,\,1\,]\mid \gamma_s\in B_{r_0}(x_0),\ \forall s\in [\,0, \,t\,[\,\}>0,\\
    t_1:&=\inf\{t\in [\,0,\,1\,]\mid \gamma_s\in B_{r_0}(x_0),\ \forall s\in \,]\,t,\, 1\,]\,\}<1.
\end{align*}
If $t_0=1$ or $t_1=0$, then $\gamma_t\in B_{r_0}(x_0)$ for all $t\in [0, 1]$, hence 
\begin{align*}
    \lvert f(\gamma_0)-f(\gamma_1)\rvert
    &\leq \lip(f) \mathsf{d}(\gamma_0, \gamma_1)
    \leq \int_0^1 G(\gamma_t)\lvert \dot \gamma_t\rvert\, \d t.
\end{align*}
Otherwise, $t_0$ and $t_1\in \,]\,0,\,1\,[$ with $f(\gamma_{t_0})=f(\gamma_{t_1})=0$, hence
\begin{align*}
    \lvert f(\gamma_0)-f(\gamma_1)\rvert
    &\leq \lvert f(\gamma_0)-f(\gamma_{t_0})\rvert+\lvert f(\gamma_{t_1})-f(\gamma_1)\rvert
    \leq \lip(f) (\mathsf{d}(\gamma_0, \gamma_{t_0})+\mathsf{d}(\gamma_{t_1}, \gamma_1))\\
    &\leq \int_0^1 G(\gamma_t)\lvert \dot \gamma_t\rvert \,\d t.
\end{align*}
If $\gamma_0\in B_{r_0}(x_0)$, $\gamma_1\not\in B_{r_0}(x_0)$, we have $t_0\in \,]\,0,\, 1\,[$ with $f(\gamma_{t_0})=f(\gamma_{1})=0$ hence
\begin{align*}
    \lvert f(\gamma_0)-f(\gamma_1)\rvert
    &= \lvert f(\gamma_0)-f(\gamma_{t_0})\rvert
    \leq \lip(f) \mathsf{d}(\gamma_0, \gamma_{t_0})
    \leq \int_0^1 G(\gamma_t)\lvert \dot \gamma_t\rvert\, \d t.
\end{align*}
A symmetric argument holds when $\gamma_0\not\in B_{r_0}(x_0)$, $\gamma_1\in B_{r_0}(x_0)$, and the same inequality is clear when $\gamma_0$, $\gamma_1\not\in B_{r_0}(x_0)$, thus integrating against any $q$-test plan for $q\in \,]\,1, \,+\infty\,]$ implies that $G\in \mathcal{UG}^{\hspace{0.03cm}p}(f)$ for any $p\in [\,1,\, +\infty\,[$, and clearly $G\in W^{1, p}(X)$. Since $\lip(X)_{\rm bs}\subset L^p(X)$, this yields the claimed inclusion.

\begin{remark}\label{rem: W1p not contained in W2p*}
     We note that although $W^{2, p}_\ast(X)$ contains all Lipschitz functions with bounded support, hence is larger than $W^{2, p}(X)$ when $(X, \mathsf{d}, \m)$ is a smooth space, we still do not expect $W^{1, p}(X)\subset W^{2, p}_\ast(X)$. Indeed, if $(X, \mathsf{d}, \m)$ is a bounded set in $\R^N$ (say, with sufficiently regular boundary and containing the origin) equipped with the usual Euclidean metric and $N$-dimensional Lebesgue measure, the function $f_\alpha(x):=\alpha^{-1}\lvert x\rvert^\alpha$ belongs to $W^{1, p}(X)$ whenever $\alpha>1-N/p$. Now suppose that $G\in \mathcal{UG}^p(f_\alpha)\cap W^{1, p}(X)$, this implies that $G(x)\geq \lvert x\rvert^{\alpha-1}$ Lebesgue a.e. on $X$. If $p>N$, by Morrey's embedding we would have that $G$ is essentially bounded, which is a contradiction if $\alpha\in \,]\,1-N/p,\, 1\,[$. If $p<N$, by the Sobolev embedding, this would imply $G\in L^{\frac{np}{N-p}}(X)$, hence $\lvert \cdot \rvert^{\alpha-1}\in L^{\frac{np}{N-p}}(X)$, however this is a contradiction if $\alpha\in \,]\,1-N/p,\, 2-N/p\,[$. Finally, if $p=N$ then $G\in L^r(X)$ for any $r\in [\,1,\, +\infty\,[$, which again yields a contradiction if $\alpha\in \,]\,1-N/p, \,1\,[$. Thus for the above ranges of $\alpha$, we have $f_\alpha\in W^{1, p}(X)\setminus W^{2, p}_*(X)$.
\end{remark}

\subsection{When there are many $H^{\hspace{0.03cm}2, p}_\ast(X)$ functions}

In this subsection we assume that $(X, \mathsf{d}, \m)$ is an  
$\mathsf{RCD}(K,N)$-space with $K\in \mathbb{R}$ and $N<\infty$, then under certain conditions we show that $H^{2, p}_*(X)$ contains many functions.

In what follows, $P_\varphi f$ will be the mollified heat flow of $f$ with respect to $\varphi\in C^\infty_c(]\,0,\, +\infty\,[)_+$ defined by \eqref{eqn: mollified heat flow}; recall that $P_{\varphi}f\in W^{1, 2}(X; \m)$ if $f\in L^2(X; \m)$, and $P_{\varphi}f\in {\rm Test}^{\infty}(X)$ if $f\in W^{1, p}(X; \m)_{\rm bs}\cap W^{1, 2}(X; \m)\cap L^\infty(X; \m)$.
\begin{lem}\label{lem: grad squared rep}
    For any $\varphi\in C^\infty_c(]\,0,\, +\infty\,[)_+$ and $f\in L^2(X; \m)$, we have
    \begin{align*}
        \lvert DP_\varphi f\rvert^2=\int_0^\infty\int_0^\infty \langle DP_tf, DP_sf\rangle \varphi(t)\varphi(s)\d t\d s,
    \end{align*}
    where the expression on the right is a Bochner integral valued in $L^1(X; \m)$; in particular the maps
    \begin{align}
    \begin{split}\label{eqn: measurable maps}
        t&\mapsto \langle DP_tf, DP_sf\rangle \varphi(t)\varphi(s),\qquad \text{each fixed }s\in \,]\,0,\,+\infty\,[,\\
        s&\mapsto \int_0^\infty \langle DP_tf, DP_sf\rangle \varphi(t)\varphi(s)\d t,
    \end{split}
    \end{align}
    are strongly measurable as $L^1(X; \m)$-valued maps.
\end{lem}
\begin{pf}
    For any fixed $g\in W^{\hspace{0.03cm}1, 2}(X; \m)$, define the operator $T_g: W^{\hspace{0.03cm}1, 2}(X; \m)\to L^1(X; \m)$ by
    \begin{align*}
        T_g(h):=\langle Dh, Dg\rangle,
    \end{align*}
    this is clearly linear. Since
    \begin{align*}
        \lVert T_g(h)\rVert_{L^1(X; \m)}
        =\int_X \lvert \langle Dh, Dg\rangle\rvert \d\m
        &\leq \lVert \lvert Dh\rvert\rVert_{L^2(X; \m)}\lVert \lvert Dg\rvert\rVert_{L^2(X; \m)}
        \\&\leq \lVert  h\rVert_{W^{\hspace{0.03cm}1,2}(X;\m)}\lVert  g\rVert_{W^{\hspace{0.03cm}1, 2}(X; \m)}
    \end{align*}
     we see $T_g$ is bounded, hence continuous. We now claim that the map $s\mapsto T_g(P_sf\varphi(s))$ belongs to $L^1(]\,0,\, +\infty\,[;$ $ \d t, L^1(X; \m))$. From the proof of Lemma~\ref{lem:denseness} we can see that $s\mapsto P_sf\varphi(s)$ is strongly measurable as a $W^{\hspace{0.03cm}1, 2}(X; \m)$-valued map, since composition with a continuous mapping preserves strong measurability, 
    $s\mapsto T_g(P_sf\varphi(s))$ is strongly measurable as an $L^1(X; \m)$-valued map. Then we calculate, using~\cite[Remark 5.2.11]{GPLecture},
    \begin{align*}
        \int_0^\infty \lVert T_g(P_sf\varphi(s))\rVert_{L^1(X; \m)}\d s
        &=\int_0^\infty \lvert\varphi(s)\rvert\int_X\lvert \langle DP_sf, Dg\rangle\rvert \d \m \d s\\
        &\leq \sup\lvert \varphi\rvert\lVert g\rVert_{W^{\hspace{0.03cm}1, 2}(X; \m)}\int_{\supp \varphi}\lVert \lvert DP_sf\rvert\rVert_{L^2(X; \m)}\,\d s\\
        &\leq\sup\lvert \varphi\rvert\lVert g\rVert_{W^{\hspace{0.03cm}1, 2}(X; \m)}\int_{\supp \varphi}\frac{\lVert f\rVert_{L^2(X; \m)}}{\sqrt{2s}}  \d s<\infty,
    \end{align*}
    proving the claim.

    Now by applying Hille's theorem (see, for example, \cite[Theorem 1.3.15]{GPLecture}) twice, we find that
    \begin{align*}
        \lvert DP_\varphi f\rvert^2
        &=T_{P_\varphi f}(P_\varphi f)
        =T_{P_\varphi f}\left(\int_0^\infty P_sf\varphi(s)\d s\right)
        =\int_0^\infty T_{P_\varphi f}(P_sf\varphi(s))\d s\\
        &=\int_0^\infty \langle DP_\varphi f, DP_sf\rangle \varphi(s)\d s
        =\int_0^\infty T_{P_s f}(P_\varphi f)\varphi(s)\d s\\
        &=\int_0^\infty T_{P_s f}\left(\int_0^\infty P_tf\varphi(t)\d t\right)\varphi(s)\d s
        =\int_0^\infty \int_0^\infty T_{P_s f}(P_tf\varphi(t))\d t\,\varphi(s)\d s\\
        &=\int_0^\infty\int_0^\infty \langle DP_tf, DP_sf\rangle \varphi(t)\varphi(s)\d t\,\d s,
    \end{align*}
    finishing the proof.
\end{pf}
\begin{prop}\label{prop: lipschitz grad bound}
    Suppose that $(X, \mathsf{d}, \m)$ is an $\mathsf{RCD}(K, N)$-space with $K\in \R$, $N\in [1, +\infty[$, and 
    \begin{align}
    &\m(X)<\infty,\label{assumption: finite measure}\\
    &\inf_{x\in X}\m(B_r(x))>0\text{ for any }r>0.\label{assumption: volume lower bound}
\end{align}
    Then for any $f\in L^2(X)$ and $\varphi\in C_c^\infty(]\,0,\, +\infty\,[)_+$, 
 we have $P_\varphi f\in W^{2, p}_\ast(X)$ for any $p\in [1, +\infty[$.
\end{prop}
\begin{pf}
Let $t_0>0$ be such that $\supp \varphi\subset ]t_0, t_0^{-1}[$. Then by applying the upper heat kernel bound in \cite[Theorem 1.2]{JiangLiZhang16} with $\varepsilon=1$ there, writing $p_t(x, y)$ for the heat kernel associated to $(P_t)_{t\geq 0}$, for some $C_1$, $C_2>0$ depending only on $K$ and $N$ we have,
\begin{align*}
    \lvert P_t f(x)\rvert
    &=\left \lvert \int_X f(y)p_{t}(x, y)\d\m(y)\right\rvert\notag\\
    &\leq \lVert f \rVert_{L^2(X)}\lVert p_{t}(x, \cdot)\rVert_{L^2(X)}\notag\\
    &\leq \lVert f \rVert_{L^2(X)}\lVert \sqrt{p_{2t}(x, x)}\\
    &\leq \frac{C_1e^{C_2t}}{\sqrt{\m(B_{\sqrt{2t}}(x))}}\lVert f \rVert_{L^2(X)}\notag\\
    &\leq \frac{C_1e^{C_2t_0^{-1}}}{\sqrt{\m(B_{\sqrt{2t_0}}(x))}}\lVert f \rVert_{L^2(X)}\notag
\end{align*}
whenever $t\in [\,t_0, \,t_0^{-1}\,]$; by~\eqref{assumption: volume lower bound}, the last expression above is bounded from above uniformly in $x\in X$ and $t\in [t_0, t_0^{-1}]$. Combining this with \cite[Proposition 6.1.6]{GPLecture}, we find another $C_3>0$ depending on $(X, \mathsf{d}, \m)$, $t_0$, and $\lVert f \rVert_{L^2(X)}$ such that
\begin{align*}
    \sup_{(t, x)\in [t_0, t_0^{-1}]\times X}\lvert DP_t f\rvert(x)
    \leq \frac{C_3}{\sqrt{t_0}},
\end{align*}
thus by Lemma~\ref{lem: grad squared rep}, we find $\lvert DP_\varphi f\rvert$ is bounded by some constant $C>0$. By~\eqref{assumption: finite measure}, constants belong to $L^p(X)$ for any $p\in [1, +\infty[$, hence have that $C\in \mathcal{UG}^p(P_\varphi f)$; again by~\eqref{assumption: finite measure} we have that constants belong to $W^{1, p}(X)$, thus $P_\varphi f\in W^{2, p}_\ast(X)$ as claimed.

\end{pf}
Since $\m(X)<\infty$, the above immediately yields the following corollary.
\begin{cor}
    Under the same conditions as Proposition~\ref{prop: lipschitz grad bound}, for any $\varphi\in C_c^\infty(]\,0, \,+\infty\,[)$ and $f\in L^2(X;\m)$, we have $P_\varphi f\in H^{\hspace{0.03cm}2, p}_\ast(X)$ for all $p\in [1, +\infty[$.
\end{cor}
\appendix
\begin{ack}
{\rm The authors would like to thank Professor Shouhei Honda for various discussions, in particular for bringing their attention to the possibility of second order Sobolev spaces only containing constant functions. 
They also thank Professor Kohei Suzuki for pointing out 
the validity of some results in the paper in the framework of $\mathsf{CD}$-spaces, when 
the second author gave a talk in this topic 
in a workshop at Durham University. 
JK was supported in part by National Science Foundation grant DMS-2246606. Part of this paper was written while JK was visiting the Department of Applied Mathematics at Fukuoka University, and he would like to thank them for their hospitality. KK was supported in part by JSPS Grant-in-Aid for Scientific Research (S) (No. 22H04942, No. 25K24482) and by fund (No. 197004) from the Central Research Institute of Fukuoka University.
}
\end{ack}

\subsection*{Conflict of interest.} 
The authors have no conflict of interest relative to the content of this article.

\subsection*{Data Availability Statement.} Data sharing is not applicable to this article as no datasets were generated or analyzed during the current study.

\subsection*{Declaration of generative AI and AI-assisted technologies in the writing process.}
No AI was used in the preparation of this manuscript.

\section{Properties of {${\rm Test}^\infty(X)$}}\label{sec: test functions}
Throughout this section, we fix an $\mathsf{RCD}(K,\infty)$-space 
$(X,\mathsf{d},\m)$ and prove some properties about the space ${\rm Test}^\infty(X)$. 

We will denote by $(P_t)_{t\geq0}$ the $\m$-symmetric semigroup on $L^2(X)$ associated with $(\mathscr{E}, D(\mathscr{E}))$, and call $(P_t)_{t\geq0}$ the \emph{heat flow}. Since $(X,\mathsf{d},\m)$ is an $\mathsf{RCD}(K,\infty)$-space, the following Bakry--\'Emery estimate (see \cite[Proposition~3.1]{GigliHan}) is known to hold:
\begin{align*}
\lvert DP_tf\rvert\leq e^{-Kt}P_t\lvert Df\rvert\quad \m\text{-a.e.}\quad \text{for} \quad f\in W^{\hspace{0.03cm}1,p}(X).
\end{align*}
Combining with \cite[Proposition 5.2.14 iii)]{GPLecture},
we have for any $f\in W^{\hspace{0.03cm}1,p}(X)$,
\begin{align}\label{eqn: P_tf W1,p est}
\lVert P_tf \lVert_{W^{\hspace{0.03cm}1,p}(X)}^p&=\lVert P_tf \lVert_{L^p(X)}^p+\lVert DP_tf \lVert_{L^p(X)}^p\notag\\
&\leq \lVert P_tf \lVert_{L^p(X)}^p+e^{-pKt}\lVert P_t\lvert Df\rvert\lVert_{L^p(X)}^p\notag\\
&\leq \lVert f\lVert_{L^p(X)}^p+e^{-pKt}\lVert \lvert Df\rvert\lVert_{L^p(X)}^p\\
&\leq\max\{1,e^{-pKt}\}\lVert f\lVert_{W^{\hspace{0.03cm}1,p}(X)}^p,\notag
\end{align}
thus $P_tf\in W^{\hspace{0.03cm}1,p}(X)$. Additionally, combining \cite[Proposition 5.2.14]{GPLecture} and \cite[Proposition 2.2]{Shigekawa97} we see $P_t$ is a strongly continuous semigroup on $L^p(X)$ for any $p\in [1, +\infty[$, that is, for any $f\in L^p(X)$,
\begin{align}\label{eqn: P_t strongly cont}
    \lim_{t\searrow 0}\lVert P_tf-f\rVert_{L^p(X)}=0.
\end{align}

We first need the following functional analysis lemma. 

\begin{lem}\label{lem:Dense}
Let $B_1, B_2$ be Banach spaces such that $B_1$ is continuously embedded into $B_2$, and 
$B_1^*, B_2^*$ their topological dual Banach spaces, respectively.    
Suppose also that $B_1$ and $B_2$ are reflexive. Then $B_2^*$ is a dense subset of $B_1^*$.  
\end{lem}
\begin{pf}
Denote by $\lVert \cdot\rVert_{B_i}$ the norm of $B_i$ ($i=1,2$). 
To prove the denseness of 
$(B_2^*,\lVert \cdot\rVert_{B_2^*})$ in $(B_1^*,\lVert \cdot\rVert_{B_1^*})$, it suffices to show the following: 
if $x\in B_1^{**}$ satisfies that $x(\ell)=0$ for any $\ell\in B_2^*$, then $x=0$ (see \cite[Chapter III, Corollary 6.14]{Conway:FunctionalAnal}). 
This is true because $x\in B_1^{**}=B_1$, which is continuously embedded into  $B_2=B_2^{**}$. 
\end{pf}
By \cite[Proposition~4.4]{GigliNobili}, it is known that $(W^{\hspace{0.03cm}1,p}(X),\lVert \cdot \lVert_{W^{\hspace{0.03cm}1,p}(X)})$ 
is a uniformly convex Banach space for $p\in\,]1,+\,\infty\,[$, hence in particular it is reflexive. Thus we can apply Lemma~\ref{lem:Dense} to find that $L^p(X)^*$ is densely embedded into $W^{\hspace{0.03cm}1,p}(X)^*$, hence
for each $\ell\in W^{\hspace{0.03cm}1,p}(X)^*$, there exists a sequence $\{\ell_n\}$ of $L^p(X)^*$ such that 
$\|\ell_n|_{W^{\hspace{0.03cm}1,p}(X)}-\ell\|_{W^{1,p}(X)^*}\to 0$ as $n\to\infty$. 
Fix $f\in W^{1,p}(X)$, then by~\eqref{eqn: P_t strongly cont} we see that for each $n\in\mathbb{N}$, $\ell_n(P_tf)\to \ell_n(f)$ as $t\to 0$.
Thus by \eqref{eqn: P_tf W1,p est} we have
\begin{align*}
|\ell(P_tf-f)|&\leq|(\ell-\ell_n|_{W^{\hspace{0.03cm}1,p}(X)})(P_tf)|+|\ell_n(P_tf-f)|+|(\ell_n|_{W^{\hspace{0.03cm}1,p}(X)}-\ell)(f)|\\
&\leq \|\ell-\ell_n|_{W^{\hspace{0.03cm}1,p}(X)}\|_{W^{1,p}(X)^*}(\|P_tf\|_{W^{\hspace{0.03cm}1,p}(X)}+\lVert f\rVert_{W^{\hspace{0.03cm}1,p}(X)})+|\ell_n(P_tf-f)|\\
&\leq 2\|\ell-\ell_n|_{W^{\hspace{0.03cm}1,p}(X)}\|_{W^{\hspace{0.03cm}1,p}(X)^*}\max\{1,e^{-Kt}\}\lVert f\rVert_{W^{\hspace{0.03cm}1,p}(X)}+|\ell_n(P_tf-f)|,
\end{align*}
and we obtain the $W^{\hspace{0.03cm}1,p}$-weak convergence of $\{P_tf\}_{t>0}$ to $f$ as $t\to0$. 
A similar argument proves the continuity of $s\mapsto P_sf$ on $[\,0,\,+\infty\,[$ with respect to 
$W^{\hspace{0.03cm}1,p}$-weak convergence. By~\cite[Lemma A.1]{GPLecture} the space $W^{\hspace{0.03cm}1,p}(X)$ is separable, thus by the $W^{\hspace{0.03cm}1,p}$-weak continuity shown above and Pettis' theorem (\cite[Chapter II.1 Theorem 2]{DiestelUhl77}), we can conclude that the $W^{\hspace{0.03cm}1,p}(X)$-valued function $s\mapsto P_sf$ is strongly measurable. Then if $\varphi\in C_c^{\infty}(]\,0,\,+\infty\,[)_+$, using \eqref{eqn: P_tf W1,p est} we obtain
\begin{align*}
\int_0^{\infty}\|P_sf\|_{W^{\hspace{0.03cm}1,p}(X)}\varphi(s)\d s&\leq \int_0^{\infty}\max\{1,e^{-Ks}\}\lVert f\rVert_{W^{\hspace{0.03cm}1,p}(X)}\varphi(s)\d s\\
&=\lVert f\rVert_{W^{\hspace{0.03cm}1,p}(X)}\int_0^{\infty}\max\{1,e^{-Ks}\}\varphi(s)\d s<\infty,
\end{align*}
which implies the Bochner integrability of $s\mapsto P_sf$ with respect to $\varphi(s)\d s$. 
The \emph{mollified heat flow of $f\in W^{\hspace{0.03cm}1,p}(X)$ (with respect to $\varphi$)} is thus defined by the $W^{\hspace{0.03cm}1,p}(X)$-valued Bochner integral 
\begin{align}\label{eqn: mollified heat flow}
P_{\varphi}f:=\int_0^{\infty}P_sf\varphi(s)\d s. 
\end{align}
Our definition here differs from that of \cite[Proposition 5.2.18]{GPLecture}, which is defined for $f\in L^2(X)\cap L^p(X)$. However, if $f\in L^2(X)\cap W^{\hspace{0.03cm}1,p}(X)$, the map $s\mapsto P_sf$ will be strongly measurable as an $L^p(X)$-valued map, hence $P_\varphi f$ will be the same object as defined in \cite[Proposition 5.2.18]{GPLecture}, thus by an abuse of notation we continue to use the same notation and terminology.

With this in mind, fix $f\in W^{\hspace{0.03cm}1,p}(X)\cap W^{\hspace{0.03cm}1,2}(X)\cap L^{\infty}(X)$, and let $\varphi_k\in C_c^{\infty}(\,]\,0,\,k^{-1}\,[\,)_+$ be such that $\int_0^{\infty}\varphi_k(s)\d s=1$ for each $k\in \N$. By considering positive and negative parts and using 
\cite[(6.9)]{GPLecture}, we have 
$P_{\varphi_k}f\in {\rm Test}^{\infty}(X)\cap W^{\hspace{0.03cm}1,p}(X)$. Now by~\eqref{eqn: P_t strongly cont}, recalling that we have strong measurability of $s\mapsto P_sf$ as an $L^p(X)$-valued map, using \cite[Chapter II.2 Theorem 4 (ii)]{DiestelUhl77} we have
\begin{align*}
\|P_{\varphi_k}f-f\|_{L^p(X)}&=\left\|\int_0^{\infty}P_sf\varphi_k(s)\d s-\int_0^{\infty}f\varphi_k(x)\d s \right\|_{L^p(X)}\\
&\leq\int_0^{k^{-1}}\|P_sf-f\|_{L^p(X)}\varphi_k(s)\d s
\leq \sup_{s\in [0, k^{-1}]}\lVert P_sf-f\rVert_{L^p(X)}\\
&\hspace{1cm}\to 0\quad\text{ as }\quad k\to\infty.
\end{align*} 
Thus, $\{P_{\varphi_k}f\}_{k\in \N}$ converges to $f$ in $L^p(X)$, and in particular 
weakly converges in $L^p(X)$.  
Applying the same argument as before~\eqref{eqn: mollified heat flow} based on Lemma~\ref{lem:Dense}, we can conclude that 
$\{P_{\varphi_k}f\}_{k\in \N}$ $W^{\hspace{0.03cm}1,p}$-weakly converges to $f$. We now apply Kakutani's theorem (see \cite{Kakutani}) on the extension of the Banach--Saks theorem to uniformly convex Banach spaces, hence we may pass to (a not-relabeled) subsequence such that the Ces{\`a}ro means of $\{P_{\varphi_k}f\}_{k\in \N}$ converge to $f$
in $W^{\hspace{0.03cm}1,p}(X)$. Thus, any element in 
$ W^{\hspace{0.03cm}1,p}(X) \cap W^{\hspace{0.03cm}1,2}(X)\cap L^{\infty}(X)$ can be $W^{\hspace{0.03cm}1,p}$-approximated by a sequence in  
${\rm Test}^{\infty}(X)\cap W^{\hspace{0.03cm}1,p}(X)$, in particular 
\begin{align}\label{eqn: test density}
    {\rm Test}^{\infty}(X)\cap W^{\hspace{0.03cm}1,p}(X)\text{ is }\lVert \cdot\rVert_{W^{\hspace{0.03cm}1,p}(X)}\text{-dense in }W^{1,p}(X)\cap W^{\hspace{0.03cm}1,2}(X)\cap L^{\infty}(X),
\end{align} 
whenever $p\in ]1, +\infty[$.

Next by utilizing \cite[Lemma 6.7]{AmbrosioMondinoSavare16}, we can find suitable cut-off functions in ${\rm Test}^{\infty}(X)_{\rm bs}$.
\begin{lem}\label{lem: localizedTest}
Suppose that $(X,\mathsf{d},\m)$ is an $\mathsf{RCD}(K,\infty)$-space  and fix $x_0\in X$. 
Then there exists a sequence $\{\zeta_{\ell}\}_{\ell\in \N}\subset {\rm Test}^{\infty}(X)_{\rm bs}$ such that $0\leq\zeta_{\ell}\leq1$ on $X$, 
$\zeta_{\ell}=1$ on $B_{\ell}(x_0)$, $\zeta_{\ell}=0$ on $B_{\ell+1}(x_0)^c$, and 
$\sup_{\ell\in\mathbb{N}}{\rm Lip}(\zeta_{\ell})<\infty$. 
\end{lem}
\begin{pf}
    For $\ell\in \N$, define $\eta_{\ell}(x):=(\ell+1-\mathsf{d}(x,x_0))^+\land 1\in \lip(X)$, satisfying  
$\eta_{\ell}=1$ on $B_{\ell}(x_0)$ and $\eta_{\ell}=0$ on $B_{\ell+1}(x_0)^c$. Thus applying the construction in the proof of \cite[Lemma 6.7]{AmbrosioMondinoSavare16} we obtain the desired sequence $\{\zeta_\ell\}_{\ell\in \N}$; note that the construction along with the Sobolev-to-Lipschitz property (\cite[Theorem 6.2]{AGS_Riem}) yields $\zeta_\ell\in \lip(X)$, with $\sup_{\ell\in\mathbb{N}}
\lip(\zeta_{\ell}) <\infty$ following from the fact that $\sup_{\ell\in\mathbb{N}}
\lip(\eta_{\ell})\leq 1$.
\end{pf}
Finally we are ready to prove our desired density result.
\begin{lem}\label{lem:denseness}
Let $(X,\mathsf{d},\m)$ be an $\mathsf{RCD}(K,\infty)$-space and $p\in \,]1, \,+\infty\,[$. 
Then ${\rm Test}^{\infty}(X)_{\rm bs}$ {\rm(}hence ${\rm Test}(X)\cap W^{\hspace{0.03cm}1,p}(X)${\rm)} is a $W^{\hspace{0.03cm}1,p}$-dense subspace 
of $W^{\hspace{0.03cm}1,p}(X)\cap W^{\hspace{0.03cm}1,2}(X)$. If $p\in[\,2,\,+\infty\,[$, then ${\rm Test}^{\infty}(X)_{\rm bs}$ {\rm(}hence ${\rm Test}(X)\cap W^{\hspace{0.03cm}1,p}(X)${\rm)} is a $W^{\hspace{0.03cm}1,p}$-dense subspace 
of $W^{\hspace{0.03cm}1,p}(X)$. 
\end{lem}
\begin{pf}
First suppose $p\in \,]1, \,+\infty\,[$. If $f\in W^{\hspace{0.03cm}1,p}(X)\cap W^{\hspace{0.03cm}1, 2}(X)$ and we set 
$f^k:=(-k)\lor f\land k$ for $k\in \N$, since for some $C>0$ independent of $\ell$,
\begin{align*}
\lvert f^k-f\rvert&\leq 2\lvert f\rvert,\qquad
    \lvert D(f^k-f)\rvert\leq 2\lvert Df\rvert,
\end{align*}
by dominated convergence we see $f^k\to f$ in $\lVert \cdot\rVert_{W^{\hspace{0.03cm}1,p}(X)}$ as $k\to\infty$; in particular $W^{\hspace{0.03cm}1,p}(X)\cap W^{\hspace{0.03cm}1, 2}(X)\cap L^{\infty}(X)$ is $W^{1, p}$-dense in $W^{\hspace{0.03cm}1,p}(X)\cap W^{\hspace{0.03cm}1, 2}(X)$. Thus by~\eqref{eqn: test density}, we see ${\rm Test}^\infty(X)\cap W^{\hspace{0.03cm}1,p}(X)$ is $W^{\hspace{0.03cm}1,p}$-dense in $W^{\hspace{0.03cm}1,p}(X)\cap W^{\hspace{0.03cm}1, 2}(X)$.

 Next fix $x_0\in X$, and let $\{\zeta_\ell\}_{\ell\in \N}$ be the sequence constructed in Lemma~\ref{lem: localizedTest}. If $f\in W^{\hspace{0.03cm}1,p}(X)$ 
 then since for some $C>0$,
 \begin{align*}
 \lvert (\zeta_{\ell}-1)f\rvert&\leq 2\lvert f\rvert,\\
\lvert D(\zeta_{\ell}-1)f\rvert&\leq\lvert \zeta_{\ell}-1\rvert\lvert Df\rvert+ \lvert f\rvert\lvert D\zeta_{\ell}\rvert\leq  \lvert Df\rvert+C\lvert f\rvert,
\end{align*}
we see $\{\zeta_\ell f\}_{\ell\in \N}\subset W^{\hspace{0.03cm}1,p}(X)_{\rm bs}$
and since after passing to a subsequence the above implies $
    \lvert (\zeta_{\ell}-1)f\rvert+\lvert D(\zeta_{\ell}-1)f\rvert\to 0$ pointwise 
$\m$-a.e., 
the dominated convergence theorem applies to yield $\zeta_{\ell}f\to f$ in $\lVert \cdot\rVert_{W^{\hspace{0.03cm}1,p}(X)}$. In particular, by taking $f\in {\rm Test}^\infty(X)\cap W^{\hspace{0.03cm}1,p}(X)$, since ${\rm Test}^\infty(X)$ is an algebra by \cite[Theorem 6.1.11]{GPLecture}, this implies that ${\rm Test}^\infty(X)_{\rm bs}$ is $W^{\hspace{0.03cm}1, p}$-dense in ${\rm Test}^\infty(X)\cap W^{\hspace{0.03cm}1,p}(X)$, finishing the proof in this case.

When $p\in\,[\,2,\,+\infty\,[$, the argument above shows $W^{\hspace{0.03cm}1,p}(X)_{\rm bs}\cap W^{\hspace{0.03cm}1, 2}(X)=W^{\hspace{0.03cm}1,p}(X)_{\rm bs}$ is dense in $W^{\hspace{0.03cm}1, p}(X)$. Thus, we immediately find that
${\rm Test}^{\infty}(X)\cap W^{\hspace{0.03cm}1,p}(X)$ is $W^{1,p}$-dense in $(W^{\hspace{0.03cm}1,p}(X),\lVert \cdot\rVert_{W^{\hspace{0.03cm}1,p}(X)})$ when $p\in\, [\,2,\, +\infty\,[$. 
\end{pf}
\bibliography{sobolev.bib}
\bibliographystyle{plain}
\end{document}